\documentclass{aic}

\aicAUTHORdetails{%
  title = {Local Girth Choosability of Planar Graphs}, 
  author = {Luke Postle and Evelyne Smith-Roberge},
  plaintextauthor = {Luke Postle, Evelyne Smith-Roberge},
  keywords = {planar graphs, list colouring, local, girth},
}   

\aicEDITORdetails{%
   year={2022},
   number={8},
   received={19 November 2021},   
   published={16 December 2022},  
   doi={10.19086/aic.2022.8},      
}   

\usepackage{amsmath,amssymb,amstext}
\usepackage{amsthm}
\usepackage{enumerate}
\usepackage{tikz}
\usepackage{thmtools}
\usepackage{thm-restate}
\usetikzlibrary{shapes}

\newtheorem{thm}{Theorem}[section]

\newtheorem{lemma}[thm]{Lemma}

\newtheorem{cor}[thm]{Corollary}
\newtheorem{claim}{Claim}

\newtheorem{obs}{Observation}

\numberwithin{equation}{section}
\theoremstyle{definition} 
\newtheorem{definition}[thm]{Definition}

\newcommand{\g}{g}
\newcommand{\Int}{\textnormal{Int}}

\begin{document}

\begin{frontmatter}[classification=text]

\title{Local Girth Choosability of Planar Graphs} 

\author[lpos]{Luke Postle \thanks{$^{,\dagger}$We acknowledge the support of the Natural Sciences and Engineering Research Council of Canada (NSERC), $^*$[Discovery Grant No.  2019-04304], $^\dagger$[CGSD3 Grant No. 2020-547516]. \newline
\hphantom{|||||} $ ^{*,\dagger}$Cette recherche a \'{e}t\'{e} financ\'{e}e par le Conseil de recherches en sciences naturelles et en g\'{e}nie du Canada (CRSNG), $^*$[Discovery Grant No.  2019-04304], $^\dagger$[CGSD3 Grant No. 2020-547516].} }
\author[esr]{Evelyne Smith-Roberge~$^\dagger$}


\begin{abstract}
In 1994, Thomassen famously proved that every planar graph is 5-choosable, resolving a conjecture initially posed by Vizing and, independently, Erd\H{o}s, Rubin, and Taylor in the 1970s. Later, Thomassen proved that every planar graph of girth at least five is 3-choosable. In this paper, we introduce the concept of a \emph{local girth list assignment}: a list assignment wherein the list size of a vertex depends not on the girth of the graph, but rather on the length of the shortest cycle in which the vertex is contained. We give a local list colouring theorem unifying the two theorems of Thomassen mentioned above. In particular, we show that if $G$ is a planar graph and $L$ is a list assignment for $G$ such that $|L(v)| \geq 3$ for all $v \in V(G)$; $|L(v)| \geq 4$ for every vertex $v$ contained in a 4-cycle; and $|L(v)| \geq 5$ for every $v$ contained in a triangle, then $G$ admits an $L$-colouring.  
\end{abstract}
\end{frontmatter}

\section{Introduction}


\subsection{Context and Main Result}
A \emph{colouring} of a graph $G$ is a function $\phi: V(G) \rightarrow C$ that assigns to each vertex $v$ of $G$ a \emph{colour} $\phi(v) \in C$ such that for every edge $uv \in E(G)$, $\phi(u) \neq \phi(v)$. If $|C| \leq k$, $\phi$ is called a \emph{$k$-colouring}; and if $G$ has a $k$-colouring, we say it is \emph{$k$-colourable}.

In 1976,  Appel and Haken \cite{kenneth1977every, appel1977every} proved the following result, settling a conjecture over a century old. It is arguably the most famous theorem in the field of graph colouring.

\begin{thm}[Four Colour Theorem \cite{kenneth1977every, appel1977every}]
Every planar graph is 4-colourable.
\end{thm}

The Four Colour Theorem is of course tight: there exist planar graphs that are not 3-colourable (for instance, $K_4$). However, in 1959 Gr\"{o}tzsch \cite{grotzsch1959dreifarbensatz} showed that a stronger bound on the chromatic number of planar graphs is obtained by forbidding triangles.

\begin{thm}[Gr\"{o}tzsch's Theorem \cite{grotzsch1959dreifarbensatz}]
Every triangle-free planar graph is 3-colourable.
\end{thm}

This paper further investigates the problem of colouring planar graphs: in particular, of \emph{list colouring} planar graphs. List colouring is a generalization of colouring introduced in the late 1970s by Vizing \cite{vizing} and, independently, by Erd\H{o}s, Rubin, and Taylor \cite{erdos1979choosability}. Given a graph $G$, a \emph{list assignment $L$} is a function that assigns to each $v \in V(G)$ a list $L(v)$ of colours.  $L$ is a \emph{$k$-list assignment} if $|L(v)| \geq k$ for every $v \in V(G)$. An \emph{$L$-colouring} of $G$ is a colouring $\phi$ such that $\phi(v) \in L(v)$ for each vertex $v\in V(G)$. $G$ is \emph{$L$-colourable} if there exists an $L$-colouring of $G$. $G$ is \emph{$k$-choosable} if $G$ is $L$-colourable for every $k$-list assignment $L$.

In 1993, Voigt \cite{voigt1993list} showed that the Four Colour Theorem does not carry over directly to list colouring by constructing a planar graph that is not 4-choosable. However, Thomassen \cite{thomassen5LC} showed in 1994 that lists of size five suffice, thus answering a conjecture posed by Vizing and Erd\H{o}s, Rubin, and Taylor in the 1970s. 
\begin{thm}[Thomassen, \cite{thomassen5LC}]\label{5choosable}
Every planar graph is 5-choosable.
\end{thm}

A natural question to ask is the following: by ruling out certain short cycles in planar graphs, can we further restrict the bound on the choosability number? Given that the answer to the analogous question for ordinary colouring is yes, it is perhaps unsurprising that this idea also works for list colouring. Recall that the \emph{girth} of a graph $G$ is the minimum number $k$ such that $G$ contains a cycle of length $k$. Since planar graphs of girth at least four (i.e.~triangle-free planar graphs) are 3-degenerate, a simple greedy argument shows that they are 4-choosable. Voigt \cite{voigt1995not} showed in 1995 that there are triangle-free planar graphs that are not 3-choosable: thus Gr\"{o}tzsch's Theorem does not carry over directly to list colouring. However, Thomassen \cite{thomassen3LC} showed in 1995 that if we also exclude 4-cycles, lists of size three suffice as follows. 

\begin{thm}[Thomassen, \cite{thomassen3LC}]\label{3choosable}
Every planar graph of girth at least five is 3-choosable.
\end{thm}

Thomassen later gave a shorter proof in 2003 \cite{thomassen3LCnew}. Though these bounds on the choosability number are tight, there is hope of strengthening Thomassen's theorems by modifying the parameters of the problem: in this paper, we introduce the concept of \emph{local girth choosability}, wherein the list size of a vertex depends not on the girth of the graph, but rather on the length of the shortest cycle in which that vertex is contained. We provide a formal definition below, following a few other necessary definitions.

\begin{definition}
Let $G$ be a graph. The \emph{girth of a vertex} $v \in V(G)$ is denoted $\g_G(v)$ and is defined as the minimum number $k$ such that $v$ is contained in a $k$-cycle. If the graph $G$ is clear from context, we will often omit the subscript and write $\g(v)$ instead of $\g_G(v)$. If $v$ is not contained in a cycle in $G$, we set $\g_G(v) = \infty$.
\end{definition}

Note that if $G' \subseteq G$ and $v \in V(G')$, then $\g_{G'}(v) \geq \g_G(v)$.
We define a \emph{local girth list assignment} as follows.

\begin{definition}
Let $G$ be a planar graph. A \emph{local girth list assignment} for $G$ is a list assignment $L$ such that:
\begin{itemize}
    \item $|L(v)| \geq 3$ for all $v \in V(G)$ with $\g(v) \geq 5$, 
    \item $|L(v)| \geq 4$ for all $v \in V(G)$ with $\g(v) = 4$, and
    \item $|L(v)| \geq 5$ for all $v \in V(G)$ with $\g(v) = 3$.
\end{itemize}
\end{definition}

We say a planar graph $G$ is \emph{local girth choosable} if $G$ admits an $L$-colouring for every local girth list assignment $L$. 
\vskip 4mm

Our main result is the following theorem.

\begin{thm}\label{localcolouring}
Every planar graph is local girth choosable.
\end{thm}

We note that Theorem \ref{localcolouring} is a joint strengthening of Theorems \ref{5choosable} and \ref{3choosable}.

The idea of restricting the list size of vertices based on the structure that surrounds them is not a new one. Indeed, many list-colouring theorems admit a \emph{local} version: that is, a version where list sizes depend on the local structure rather than a global property of the graph. For instance, Borodin, Kostochka and Woodall \cite{borodin1997list} proved a local version of Galvin's Theorem that the List Colouring Conjecture holds for bipartite graphs, where the size of an edge's list depends on the maximum degree of its endpoints. In a similar vein, Bonamy, Delcourt, Lang, and Postle \cite{bonamy2020edge} proved a local asymptotic version of Kahn's Theorem on list-edge colouring. In \cite{kelly2020local}, Kelly and Postle proved a local epsilon version of Reed's conjecture, where list sizes are lower-bounded by a linear combination of vertices' degrees and the size of the largest clique in which they are contained.   In \cite{davies2020coloring}, Davies, de Joannis de Verclos, Kang, and Pirot gave an asymptotic theorem for list-colouring triangle-free graphs, where again vertices' list sizes are bounded by a function of their degree.

The proof of our main theorem is inspired by those of Theorems \ref{5choosable} and \ref{3choosable}. To prove those theorems, Thomassen instead proved more technical theorems (Theorems \ref{tech5choos} and \ref{tech3choos}, given later on in this section) involving precolouring paths on the outer face boundary of plane graphs, and extending these to colourings to the whole graph. In Theorems \ref{tech5choos} and \ref{tech3choos}, list sizes are restricted still further than in Theorems \ref{5choosable} and \ref{3choosable}: in particular, vertices on the outer face boundary of the graphs have smaller lists than the others. It is precisely these stronger restrictions that allow the theorems to be proven inductively: a subset of the vertices of a vertex-minimum counterexample to each theorem are coloured, their colours removed from the lists of neighbouring vertices, and the coloured vertices deleted.  If this is done carefully, the resulting graph will still satisfy the premises of the technical theorem, allowing the proof to be completed by induction. Like Thomassen, we also prove our main colouring result (Theorem \ref{localcolouring}) via a more technical theorem. Our main technical theorem is Theorem \ref{345colouring}, which implies both of Thomassen's technical theorems, Theorems \ref{tech5choos} and \ref{tech3choos}, and hence implies both of Theorems \ref{5choosable} and \ref{3choosable}.

Unfortunately, the proofs of Theorems \ref{5choosable} and \ref{3choosable} cannot be readily combined to prove Theorem \ref{localcolouring}. The proof of Theorem \ref{tech5choos} (which implies Theorem \ref{5choosable}) relies on the fact that we may delete up to two colours in the lists of vertices not on the outer face boundary of the graph without making their lists too small. The same is not true for a graph with a local girth list assignment. The proof of Theorem \ref{tech3choos} (which implies Theorem \ref{3choosable}) of course relies on the fact that every vertex in the graph has girth at least five: this ensures that in colouring and deleting a path of length at most one on the outer face boundary of the graph, the vertices that were adjacent to these deleted vertices form an independent set (a condition required for the technical inductive statement to go through).  

Though many of our lemmas are similar in spirit to those used by Thomassen, the main colouring argument is quite different: in our proof, we colour and delete an \emph{arbitrarily long} path on the outer face boundary of the graph. The path in question is determined by the lists of the vertices on the outer face boundary. Since the structure of the graph under study is more complex than that of a simple near-triangulation or a graph of girth at least five, a stronger inductive statement and more structural lemmas are needed than for either of the proofs  of Theorems \ref{5choosable} or \ref{3choosable}. Indeed for our proof, we even have to generalize Thomassen's characterization of when a precoloured path of length two does not extend to a colouring (Theorem \ref{thomassen3ext} below) to the local setting. Unfortunately then, the stronger inductive statement given in Theorem \ref{345colouring} gives rise to a number of exceptional graphs which render the analysis more difficult: when working through inductive arguments, we will have to argue why these exceptional cases do not arise. 

It is perhaps surprising that there is a single theorem that unifies both Theorems \ref{5choosable} and \ref{3choosable}. After all, the proofs of the Four Colour Theorem and Gr\"{o}tzsch's Theorem \textemdash which are in a sense the analogous theorems for ordinary vertex colouring\textemdash are very different. That Theorem \ref{localcolouring} (a local, unified version for list colouring) holds hints perhaps at something fundamental about planar graphs and the relation between colouring and short cycles. 

\subsection{Preliminaries and Technical Result}
Before we present our more technical theorem, we give a few required definitions.

The \emph{outer} face of a plane graph is its infinite face. Let $G$ be a plane graph, and let $E$ and $V$ be the set of edges and vertices, respectively,  in the boundary walk of the outer face of $G$. Let $H$ be the subgraph of $G$ with $V(H) = V$ and $E(H) = E$.  We say a path (or cycle) $S$ is \emph{on the outer face boundary of $G$} if $S \subseteq H$.  

\begin{definition}
Let $G$ be a plane graph, and let $S = v_1v_2 \dots v_k$ be a path in $G$. We say $S$ is an \emph{acceptable path} if one of the following holds. 
\begin{itemize}
    \item $k \leq 3$, or
    \item $k = 4$, $\g(v_2) \geq 4$, and $\g(v_3) \geq 4$, or
    \item $k = 4$, and either $\g(v_2) \geq 5$ or $\g(v_3) \geq 5$.
\end{itemize}
An \emph{acceptable cycle} is a cycle $S$ in $G$ where $S$ contains an edge $e$ such that $S-e$ is an acceptable path in $G$. 
\end{definition}

Note that in the above definition, we allow $S$ to be the empty path. Moreover, note that if $S$ is an acceptable path in a graph $G$, $S_1$ is a subpath of $S$, and $G_1$ is a subgraph of $G$ that contains $S_1$, then $S_1$ is an acceptable path of $G_1$.

Theorem \ref{345colouring} below \textemdash the proof of which constitutes the bulk of this paper \textemdash characterizes precisely when a colouring of an acceptable path or cycle extends to a list colouring of the whole graph. The list assignment described in the theorem is a restriction of a local girth list assignment. Under this list assignment, the list sizes depend in part on whether or not the vertices are on the outer face boundary of the graph.  For this reason (among others), the following definition is useful.

\begin{definition}
Let $G$ be a plane graph, and let $C \subseteq G$ be a cycle. $\Int(C)$ is defined as the subgraph of $G$ induced by the vertices of $G$ in the interior of $C$. Similarly, we define $\Int[C]$ to be the subgraph of $G$ containing precisely the vertices and edges inside and on $C$.
\end{definition}

If the graph under study contains one of a specific set of subgraphs, a colouring of an acceptable path is only guaranteed to extend to a colouring of the whole graph if the acceptable path is short. Among these problematic subgraphs are the \emph{broken} and \emph{generalized wheels}, defined below. We take the definitions and associated terminology from Thomassen in \cite{thomassen2007exponentially}.
\begin{definition}
A \emph{broken wheel} is either a cycle $v_1v_2v_3v_1$ or the graph formed by a single cycle $v_1v_2 \dots v_qv_1$ together with edges $v_2v_4, v_2v_5, \dots, v_2v_q$. The path $v_1v_2v_3$ is called the \emph{principal path} of the broken wheel. A \emph{wheel} is a cycle $v_1v_2 \dots v_qv_1$ together with a single vertex $v$ and edges $vv_1,vv_2, \dots, vv_q$. Again, we say $v_1v_2v_3$ is the \emph{principal path} of the wheel.  If $v_1v_2v_3$ is the principal path of a wheel or broken wheel, we call $v_1v_2$ and $v_2v_3$ the principal edges.
\end{definition}

\begin{definition}
A \emph{generalized wheel} is defined recursively as follows. Every wheel and broken wheel is a generalized wheel. If $W$ is a generalized wheel with principal path $v_1v_2v_3$ and $W'$ is a generalized wheel with principal path $u_1u_2u_3$, then the graph obtained from $W$ and $W'$ by identifying a principal edge of $W$ with a principal edge in $W'$ in such a way that $v_2$ and $u_2$ are identified is also a generalized wheel. Its principal path is the path formed by the principal edges in $W$ and $W'$ that were not identified. 
\end{definition}

Note that if $W$ is a generalized wheel and $C$ is the outer cycle of $W$, then either $C$ has a chord, $W$ is a wheel, or $W$ is a 3-cycle. Moreover, note that as there may be multiple ways to construct a generalized wheel $W$, there maybe several possible principal paths for $W$. Unless $W$ is a triangle or a wheel, these paths are edge-disjoint. Finally, note that every vertex in a generalized wheel has girth 3.

As mentioned earlier, Theorem \ref{345colouring} describes when a list colouring of an acceptable path $S$ in a graph $G$ extends to a list colouring of the entire graph. Below, we define a \emph{canvas}, which will allow us to concisely keep track of the graph, acceptable path, and list assignment.

\begin{definition}
Let $G$ be a plane graph, and let $C$ be the subgraph whose vertex- and edge-set are precisely those of the outer face boundary of $G$. We say  $(G, L, S, A)$ is a \emph{canvas} if $S$ is a subgraph of  $C$, $A \setminus V(S)$ is an independent set of vertices in $V(C)\setminus V(S)$ such that $\g(v) \geq 5$ for each $v \in A$, and $L$ is a list assignment whose domain contains $V(G)$ such that: 
\begin{itemize}
    \item $|L(v)| \geq 1$ for all $v \in V(S)$,
    \item $|L(v)| = 2$ for all $v \in A \setminus V(S)$,
    \item $|L(v)| \geq 3$ for all $v \in V(G) \setminus (A \cup V(S))$,
    \item $|L(v)| \geq 4$ for all $v \in V(G) \setminus V(C)$ such that $\g(v) = 4$, and 
    \item $|L(v)| \geq 5$ for all $v \in V(G) \setminus V(C)$ such that $\g(v) = 3$.
\end{itemize}
We say $(H, L, S \cap H, A \cap V(H))$ is a \emph{subcanvas} of a canvas $K = (G, L, S, A)$ if $H \subseteq G$. In this case, we denote $(H, L, S\cap H, A \cap V(H))$ by $K[H]$. Note that $K[H]$ is a canvas.
\end{definition}

We think of the vertices of $S$ as being \emph{precoloured}, and we say a canvas $K = (G, L, S, A)$ admits an $L$-colouring if $G$ admits an $L$-colouring. It might seem more natural to require that $A$ (and not $A \setminus V(S)$) is an independent set of vertices with lists of size two; this definition is more convenient, since when working through inductive arguments the precoloured vertices are sometimes elements of $A$. Having established the required definitions, we give our technical theorem below. 


\begin{thm}\label{345colouring}
Let $G$ be a plane graph, and let $S$ be either an acceptable path $v_1v_2 \cdots v_k$ or acceptable cycle $v_1v_2\cdots v_kv_1$. If $(G, L, S, A)$ is a canvas, then every $L$-colouring $\phi$ of $G[V(S)]$ extends to an $L$-colouring of $G$ unless one of the following occurs: 
\begin{enumerate}[(i)]
    \item $k=4$; there exists a vertex $u \in A \setminus V(S)$ that is adjacent to both $v_1$ and $v_4$; and $L(u) = L(v_1) \cup L(v_4)$, or
    \item $k=4$; there exists a vertex $u \in A \setminus V(S)$ such that, up to reversing the names of the vertices of $S$, $u$ is adjacent to $v_4$; there exists a vertex $w \not \in V(S)$ on the outer face boundary of $G$ such that $uw \in E(G)$;  $v_1v_2w$ is the principal path of a generalized wheel $W$ where the vertices on the outer cycle of $W$ are on the outer face boundary of $G$; every vertex $v$ in the outer cycle of $W$ except $v_1$ and $v_2$ has $|L(v)| = 3$, or
    \item $k=3$; $S$ is the principal path of a generalized wheel $W$ where the vertices on the outer cycle of $W$ are on the outer face boundary of $G$; every vertex $v$ in the outer cycle of $W$ except $v_1, v_2,$ and $v_3$ has $|L(v)| = 3$.
\end{enumerate}
\end{thm}

We say  $(G, L, S, A)$ is an \emph{exceptional canvas of type (i), type (ii),} or \emph {type (iii)}, if it is described by (i), (ii), or (iii), respectively, above. If a canvas is not an exceptional canvas of any type, we say it is \emph{unexceptional}.  Note that an exceptional canvas of type (ii) or (iii) might still admit a colouring, depending on $L$. Examples of exceptional canvases that do not admit an $L$-colouring are given in Figure \ref{fig:bad_cases}. We note that there is precisely one case where $(G,L,S,A)$ is unexceptional, there exists an $L$-colouring of $S$ but not necessarily of $G[V(S)]$: this is the case where $S$ is a path of length three and $V(S)$ induces a 4-cycle.

While Theorem \ref{345colouring} is not an if and only if statement, it is rather straightforward to characterize the list assignments of exceptional canvases that do not admit a colouring (see \cite{thomassen2007exponentially} and \cite{postlethomas3} for such details).

As a corollary to Theorem \ref{345colouring}, we immediately obtain Theorem \ref{localcolouring}: to see this, note that if $(G, L, S, A)$ is a canvas, then by definition $L$ is a restriction of a local girth list assignment. Moreover, if $|V(S)| \leq 2$ (that is, if at most two vertices on the outer face boundary of $G$ are precoloured), then $(G, L, S, A)$ is unexceptional and so $G$ admits an $L$-colouring. 

\begin{figure}[ht]
\begin{center}
\begin{tikzpicture}[scale=0.80]
   \newdimen\R
   \R=2cm
   \draw (0:\R) \foreach \x in {72,144,...,360} {  -- (\x:\R) };
   \foreach \x/\l/\p in
     { 72/{\{2\}}/above,
      144/{\{1\}}/left,
      216/{\{2\} }/left,
      288/{\{1\}}/below
     }
     \node[inner sep=1pt, minimum size=7pt, circle,draw,fill,label={\p:\l}] (\x) at (\x:\R) {};
   \node[inner sep=1pt, minimum size=7pt, star,star points=4,star point ratio=0.5, draw, fill=white, label={right:\{1,2\}}] (360) at (360:\R) {}; 

\end{tikzpicture}
\hskip 6mm
\begin{tikzpicture}[scale=0.72, invisnode/.style={circle, draw=white, fill=white,  minimum size=1mm}]
\node[inner sep=1pt, minimum size=7pt, circle,draw,fill, label={right:\{2\}}] at (1,0) (v1) {};
\node[inner sep=1pt, minimum size=7pt, circle,draw,fill, label={right:\{4\}}] at (0.2,2) (v2) {};
\node[inner sep=1pt, minimum size=7pt, circle,draw,fill, label={left:\{3\}}] at (-2.2,2) (v3) {};
\node[inner sep=1pt, minimum size=7pt, circle,draw,fill, label={left:\{2\}}] at (-3,0) (v4) {};
\node[inner sep=1pt, minimum size=7pt, star,star points=4,star point ratio=0.5, draw, fill=white, label={right:\{1,2\}}] (v5) at (0.3, -2) {};
\node[inner sep=1pt, minimum size=7pt, circle,draw,fill=white, label={below:\{1,2,3\}}] at (-1,-2) (v6) {};
\node[invisnode] at (-3,-1) (v7) {};
\node[invisnode] at (-2, -2) (v8) {};
\node[] at (-2.5, -1.4) {\Huge{$\ddots$}};
\node[] at (-2.2, -0.5) {\Large{$W$}};

\draw[black] (v1)--(v2);
\draw[black] (v2)--(v3);
\draw[black] (v3)--(v4);
\draw[black] (v1)--(v5);
\draw[black] (v5)--(v6);
\draw[black] (v3)--(v6);
\draw[black] (v4)--(v7);
\draw[black] (v6)--(v8);    
\end{tikzpicture}
\hskip 6mm
\begin{tikzpicture}[scale=0.72, invisnode/.style={circle, draw=white, fill=white,  minimum size=1mm}]
\node[inner sep=1pt, minimum size=7pt, circle,draw,fill, label={right:\{1\}}] at (0.5,0) (v1) {};
\node[inner sep=1pt, minimum size=7pt, circle,draw,fill, label={above:\ \ \ \{2\}}] at (0,2) (v2) {};
\node[inner sep=1pt, minimum size=7pt, circle,draw,fill, label={above:\{3\}}] at (-2,2.3) (v3) {};
\node[invisnode] at (-4,0.5) (v4) {};
\node[invisnode] at (-1, -2) (v5) {};
\node[] at (-3, -1) {\Huge{$\ddots$}};
\node[] at (-1.5, 0.2) {\Large{$W$}};

\draw[black] (v1)--(v2);
\draw[black] (v2)--(v3);
\draw[black] (v3) to [out=180,in=80,looseness=1] (v4);
\draw[black] (v1) to [out=-90,in=20,looseness=1] (v5);    
\end{tikzpicture}
\end{center}
\caption{Exceptional canvases of type (i), (ii), and (iii), together with potential partial list assignments. Vertices in $S$ are black; vertices in $A$ are drawn as four-pointed stars. $W$ indicates the presence of a generalized wheel subgraph.}
    \label{fig:bad_cases}
\end{figure}
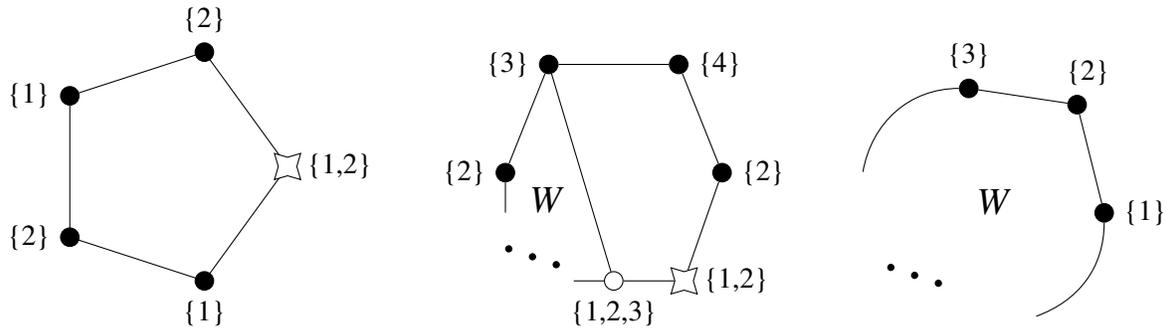

We highlight the similarity between our technical theorem and the technical theorems used by Thomassen in proving Theorems \ref{5choosable} and \ref{3choosable}: these technical theorems are rephrased below using our terminology.

\begin{thm}[Thomassen \cite{thomassen5LC}]\label{tech5choos} If $(G, L, S, A)$ is a canvas where $G$ is a near-triangulation and $S$ is a path of length at most one in the outer face boundary of $G$, then $G$ admits an $L$-colouring.
\end{thm}
This theorem implies Theorem \ref{5choosable}, that every planar graph is 5-choosable, and is implied by Theorem \ref{345colouring}. We note that in a near triangulation, every vertex has girth three. Consequently, in the above theorem $A = \emptyset$ by definition.

Thomassen also characterized the set of canvases $(G, L, S, A)$ where $G$ is a near-triangulation and $S$ a path of length two that do \emph{not} necessarily admit an $L$-colouring. Thomassen's original theorem is Theorem 3 in \cite{thomassen2007exponentially}. We state it below in the language of canvases.

\begin{thm}[Thomassen \cite{thomassen2007exponentially}]\label{thomassen3ext}
If $K = (G, L, S, A)$ is a canvas where $G$ is a near-triangulation and $S$ is a path of length two in the outer face boundary of $G$, then $G$ admits an $L$-colouring unless $K$ is an exceptional canvas of type (iii).
\end{thm}

We note that Theorem \ref{thomassen3ext} is also a special case of Theorem \ref{345colouring} where $G$ is a near-triangulation (and so $A= \emptyset$).

One technical theorem that may be used to establish Theorem \ref{3choosable} is as follows. 

\begin{thm}\label{tech3choos}
If $K=(G, L, S, A)$ is a canvas where $S$ is an acceptable path and every vertex in $G$ has girth at least $5$, then every $L$-colouring of $S$ extends to an $L$-colouring of $G$ unless $K$ is an exceptional canvas of type (i).
\end{thm}

Thomassen's true technical theorem forbids edges between $S$ and $A$, but allows $S$ to have up to six vertices. However, as mentioned in the conclusion of \cite{thomassen3LCnew}, this is essentially equivalent to Theorem \ref{tech3choos}. Note that Theorem \ref{tech3choos} is a special case of Theorem \ref{345colouring}, and it implies Theorem \ref{3choosable} (and moreover Gr\"{o}tzsch's Theorem, as is shown in \cite{thomassen3LCnew}).

A specific type of exceptional canvases of type (iii) appears frequently in the induction arguments in our paper. The following lemma will allow us to deal with them painlessly. It is Lemma 1 by Thomassen in \cite{thomassen2007exponentially}. The proof is approximately two thirds of a page; we omit it in the interest of brevity.

\begin{lemma}[Thomassen \cite{thomassen2007exponentially}]\label{thomassenbadcol}
Let $W$ be a generalized wheel but not a broken wheel with outer cycle $C = v_1v_2v_2 \dots v_q$. If $L$ is a list assignment for $W$ such that every vertex $v \in \Int(C)$ has $|L(v)| \geq 5$ and every very vertex $v \in V(C) \setminus \{v_1, v_2, v_3\}$ has $|L(v)| \geq 3$, then there is at most one colouring of $v_1, v_2, v_3$ that does not extend to an $L$-colouring of $G$. 
\end{lemma}

Note that Lemma \ref{thomassenbadcol} and Theorem \ref{345colouring} together give us the following lemma. It is an analogous lemma to Lemma \ref{thomassenbadcol} for local girth list assignments. 
\begin{lemma}\label{bad-colouring}
Let $G$ be a plane graph, and let $S = v_1v_2v_3$ be a path on the outer face boundary of $G$. Suppose $G$ contains a subgraph $W$ that is a generalized wheel whose principal path is $v_1v_2v_3$ such that the vertices on the outer cycle of $W$ are on the outer face boundary of $G$. Let $(G, L, S, A)$ be a canvas. If $W$ is not a broken wheel, then there is at most one $L$-colouring of $G[V(S)]$ that does not extend to an $L$-colouring of $G$.
\end{lemma}

We will not use Lemma \ref{bad-colouring} in our paper: our inductive arguments will allow us to use Theorem \ref{thomassenbadcol} directly. 

We require one more definition.
\begin{definition}\label{seppathdef}
Let $G$ be a connected plane graph with outer face boundary walk $C$. Let $P$ be a path with endpoints $u,v$, where $\{u,v\} \subseteq V(C)$. We say \emph{$P$ separates $G$ into two plane graphs $G_1$ and $G_2$} if the following hold.
\begin{itemize}
    \item $G_1$ and $G_2$ inherit the embedding of $G$,
    \item $G_1 \cap G_2 = P$,
    \item $G_1 \cup G_2 = G$, 
    \item for each $i \in \{1,2\}$, we have that $V(G) \setminus V(G_i) \neq \emptyset$, 
    \item both $G_1$ and $G_2$ are connected, and
    \item $P$ is in the outer face boundary walk of both $G_1$ and $G_2$.
\end{itemize}
\end{definition}
\subsection{Outline of Paper}
Theorem \ref{345colouring} is proved in Section 3, by supposing the existence of a minimum counterexample. To that end, for the remainder of the paper we let $K=(G, L, S, A)$ be a counterexample to Theorem \ref{345colouring}, chosen such that $|V(G)|$ is minimized over all counterexamples of Theorem \ref{345colouring}; and subject to that, such that $\sum_{v \in V(G)} |L(v)|$ is minimized. In Section 2, we establish necessary structural properties of $K$. Note that by our choice of $K$, it follows that $S$ contains at least one vertex. We may assume that there exists at least one colouring of $G[V(S)]$ that does not extend to a colouring of $G$, as otherwise there is nothing to prove. Since $K$ is chosen to minimize $\sum_{v \in V(G)} |L(v)|$, it follows that $|L(v)| = 1$ for each vertex $v \in V(S)$. To derive a contradiction (and hence prove Theorem \ref{345colouring}), it thus suffices to show that $G$ admits an $L$-colouring.

\section{A Description of Our Minimum Counterexample}

Recall that $K = (G, L, S, A)$ is a counterexample to Theorem \ref{345colouring}. We may assume without loss of generality that $V(S) \cap A = \emptyset$. Let $C$ be the subgraph of $G$ whose vertex set and edge set is precisely that of the outer face boundary walk of $G$.  It is straightforward to verify the theorem holds if $|V(G)| = |V(S)|$, so we assume $|V(G)| > |V(S)|$. We note the following.

\begin{obs}\label{subcanvas}
If $K$ is an unexceptional canvas and $K'$ is a subcanvas of $K$, then $K'$ is also unexceptional. 
\end{obs} 

This follows from the fact that whether or not a canvas is exceptional depends in part on $L$, and that every vertex in a generalized wheel has girth three: thus if the outer face boundary of $K'$ contains a vertex $v$ not in the outer face boundary of $K$ with $\g(v) = 3$, then $|L(v)| \geq 5$. We will use this observation repeatedly throughout the paper.

\subsection{Section Overview}
We now give an overview of the results of this section. As this subsection is purely a summary of this section's results compiled for the reader's convenience, the reader should feel free to skip ahead to Subsection \ref{groundworklemmas} if desired.

In Subsection \ref{groundworklemmas}, we give a few lemmas establishing basic properties of $K$: namely, Lemma \ref{2conn} shows that $G$ is 2-connected (and thus, that $C$ is a cycle); and Lemma \ref{chordless}, that its outer cycle $C$ is chordless. These first two lemmas allow us to make useful assumptions about the structure of $G$ later on. 

The majority of Subsection \ref{separatingcyclelemmas} describes the interiors of certain cycles in $G$ of length up to six. In particular, Lemma \ref{sep3cycle} shows that $G$ does not contain cycle of length three with vertices in their interior. Observation \ref{listsizes} argues that we may assume no vertex in $C$ has more than three colours in its list. This allows us to argue explicitly about the lists themselves in Subsection \ref{towardsadeletablepath} and in our final colouring argument in Section 3. Lemma \ref{sep4cycle} shows that $G$ does not contain cycles of length four with vertices in their interior. Observation \ref{kgeq3} shows that $|V(S)| \geq 3$. Lemma \ref{sep5cycle} shows that if $H \subseteq G$ is a 5-cycle with at least one vertex in its interior, then all vertices in $V(H)$ have girth three; and finally, Lemma \ref{sep6cycle} shows that if $H \subseteq G$ is a 6-cycle with at least one vertex in its interior, then no vertex in $V(H)$ has girth at least five. These lemmas are used to establish more complicated structural properties of $G$ (in particular, to rule out the existence of certain generalized wheel subgraphs in Subsection \ref{brokenwheellemmas}). 

In Subsection \ref{seppathlemmas}, we characterize the set of short paths in $G$ whose endpoints lie in $C$ and whose interior vertices do not. We think of these paths as \emph{separating} paths: such a path neatly divides $G$ into two subgraphs whose intersection is the path (see Definition \ref{seppathdef}). In particular, we show in Lemma \ref{intg4s} that $G$ does not contain a separating path of length two that has at least one vertex of girth at least four, and in Lemma \ref{3-5-5-3} we describe exactly the structure surrounding separating paths $u_1u_2u_3u_4$ where both $u_2$ and $u_3$ have girth at least five. 

Subsection \ref{brokenwheellemmas} uses the structure established in previous subsections to rule out the existence of certain generalized wheel subgraphs in $G$, and to bound the size of others. We note that whether or not a generalized wheel subgraph can be ruled out depends not only on its structure, but also its list assignment. Lemma \ref{nogenwheels} shows that if $Q$ is a separating path of length two whose vertices all have girth three, then $Q$ is the principal path of a broken wheel in $G$. Lemma \ref{closeafan} shows that upon identifying certain vertices in a generalized wheel, the girth of the remaining vertices in the graph do not change too much: every vertex with girth at least five in $G$ still has girth at least five after the identification, and every vertex with girth four in $G$ still has girth at least four after the identification. This allows us in Lemma \ref{nobigfans} to rule out some large broken wheels in $G$. Lemma \ref{deg3forv4} rules out yet another type of broken wheel in $G$, and will be used in Section 3 to show that, after performing our main reductions, what remains is not an exceptional canvas of type (iii). 

Finally, Subsection \ref{towardsadeletablepath} establishes the remainder of the structure required for the proof of Theorem \ref{345colouring} in the next section. In particular, Lemma \ref{extraverts1} shows that $C$ contains at least three vertices other than those in $S$, and Lemma \ref{extraverts2} shows that if there is a vertex in $A$ adjacent to a vertex in $S$, then this lower bound can be raised to four. Lemma \ref{colours} describes the lists of these non $S$-vertices. Lastly, we give the definition of a particular type of path contained in $C$ called a \emph{deletable path}: this path comprises the main reducible configuration for the proof of Theorem \ref{345colouring} in Section \ref{345colouring}. We end the section by showing that $K$ contains a deletable path.

\subsection{Groundwork Lemmas}\label{groundworklemmas}

We begin by showing that $G$ is 2-connected. Note that $G$ is connected: otherwise, since $K$ is a minimum counterexample to Theorem \ref{345colouring} we obtain an $L$-colouring of $G$ by applying Theorem \ref{345colouring} to each component of $G$. Moreover, since $K$ was chosen to minimize $\sum_{v \in V(G)} |L(v)|$, we may assume without loss of generality that $|V(S)| \geq 2$. Since $G[V(S)]$ has an $L$-colouring by assumption, we may further assume that $|V(G)| \geq 3$. 

\begin{lemma}\label{2conn}
$G$ is 2-connected.
\end{lemma}
\begin{proof}
Suppose not. Since $|V(G)| \geq 3$, there exists a cut vertex $u \in V(G)$ such that the path $u$ separates $G$ into two graphs $G_1$ and $G_2$ (as in Definition \ref{seppathdef}). For $i \in \{1,2\}$, let $S_i$ be the subgraph of $S$ contained in $G_i$. Suppose without loss of generality that $|V(S_1)| \geq |V(S_2)|$. Since $S$ is a path or cycle in the outer face boundary of $G$, it follows that either $V(S_1) \cap V(S_2) = \emptyset$, or $V(S_1) \cap V(S_2) = \{u\}$. By Observation \ref{subcanvas}, $K[G_1]$ is an unexceptional canvas. By the minimality of $K$, $G_1$ thus admits an $L$-colouring $\phi_1$. Let $L'$ be a list assignment for $G_2$ obtained by setting $L'(v) = L(v)$ for all $v \in V(G_2) -u$ and setting $L'(u) = \{\phi(u)\}$. Note that since $S_2 \subset S$ and $S$ is an acceptable path or cycle, $S_2$ is an acceptable path. Since $|V(S_1)| \geq |V(S_2)|$, it follows that $|V(S_2)| \leq 2$ and so that $(G_2, L', S_2, A \cap V(G_2))$ is an unexceptional canvas. Thus $G_2$ admits an $L'$-colouring $\phi_2$. By construction, $\phi_1$ and $\phi_2$ agree on $u$, and so $\phi_1 \cup \phi_2$ is an $L$-colouring of $G$, a contradiction. 
\end{proof}

Since $G$ is 2-connected by Lemma \ref{2conn}, it follows that $C$ is a cycle. We now show this cycle is chordless.




\begin{lemma}\label{chordless}
$C$ is chordless.
\end{lemma}
\begin{proof}
Suppose not, and let $uw$ be a chord of $C$.  Thus the path $uw$ separates $G$ into two graphs $G_1$ and $G_2$. Suppose $S$ is a cycle. Since $C$ has a chord, $|V(C)| \geq 4$. Thus $S$ is a 4-cycle. But then at least one vertex in $S$ has girth at least four by the definition of acceptable cycle. Thus $C$ is chordless \textemdash a contradiction. It follows that $S$ is a path.

Suppose now that $S \subseteq G_1$. Note that since $K$ is unexceptional, it follows from Observation \ref{subcanvas} that $K[G_1]$ is also unexceptional.   Since $K$ is a minimum counterexample to Theorem \ref{345colouring} and $|V(G_1)| < |V(G)|$, it follows further that $K[G_1]$ admits an $L$-colouring $\phi$. Let $L'$ be a list assignment for $G_2$ obtained from $L$ by setting $L'(v) = L(v)$ for all $v \in V(G_2)\setminus \{u,w\}$, and $L'(v) = \{\phi(v)\}$ for $v \in \{u,w\}$. Note that $S \cap V(G_2) \subseteq \{u,w\}$, and so that $K_2 = (G_2, L', uw, A \cap V(G_2))$ is a canvas. Moreover, $uw$ is an acceptable path for $G_2$ since it contains only two vertices. This also implies that $K_2$ is unexceptional. Since $|V(G_2) | < |V(G)|$, we have by the minimality of $K$ that $G_2$ admits an $L'$-colouring $\phi'$. This is a contradiction since $\phi \cup \phi'$ is an $L$-colouring of $G$. 

Thus $S \not \subseteq G_1$, and symmetrically $S \not \subseteq G_2$. Suppose there is no vertex of $S$ contained in $V(G_2) \setminus V(G_1)$. Then since $S \not \subseteq G_1$, it follows that there is an edge $e$ of $S$ in $E(G_2)\setminus E(G_1)$ but that both endpoints of $S$ are in $G_1$. Since $G_1 \cap G_2 = uv$, it follows that $e=uv$, a contradiction since $uv\in E(G_1)$. Thus there exists a vertex of $S$ contained in $V(G_1) \setminus V(G_2)$, and symmetrically, a vertex of $S$ contained in $V(G_2) \setminus V(G_1)$. It follows that $|V(S)| \geq 3$, and one of $u$ and $w$ is an internal vertex of the path $S$. We may assume without loss of generality that $w$ is an interior vertex of $S$. Note that $u$ is not contained in $V(S)$ since otherwise $|V(S)| = 4$ and $S$ contains a subpath of length 2 containing only vertices of girth 3, contradicting the definition of acceptable path. Let $S_1$ and $S_2$ be subpaths of $S$ such that $|V(S_1)| \leq |V(S_2)|$; $S_1 \cup S_2 = S$; and $S_1 \cap S_2 = w$. Note that both $S_1$ and $S_2$ are acceptable paths in $G$, since a subpath of an acceptable path is itself acceptable. We may assume without loss of generality that $S_1 \subseteq G_1$ and $S_2 \subseteq G_2$.

First suppose that $|L(u)| \geq 4$. Since $K[G_1]$ is a subcanvas of $K$ and $K$ is unexceptional, it follows from Observation \ref{subcanvas} that $K[G_1]$ is also unexceptional. By the minimality of $K$, $K[G_1]$ admits an $L$-colouring $\phi_1$. Let $L'$ be a list assignment for $G_1$ obtained by setting $L'(v) = L(v)$ for all $v \in V(G_1)-u$, and setting $L'(u) = L(u) \setminus \{\phi_1(u)\}$. Since $|L(u)| \geq 4$, it follows that $|L'(u)|\geq 3$ and so that $(G_1, L', S_1, A\cap V(G_1))$ is a canvas. Note that since $|V(S_1)| \leq |V(S_2)|$ and $|V(S)| \leq 4$, we have that $|V(S_1)| \leq 2$.  Thus $(G_1, L', S_1, A\cap V(G_1))$ is unexceptional. By the minimality of $K$, there exists an $L'$-colouring $\phi_2$ of $(G_1, L', S_1, A\cap V(G_1))$. Let $L''$ be a list assignment for $G_2$ obtained by setting $L''(v) = L(v)$ for all $v \in V(G_2)-u$, and $L''(u) = \{\phi_1(u), \phi_2(u), \phi_1(w)\}$. Note that since $S$ is an acceptable path and $w$ is an interior vertex of $S$, either $|V(S_2)| = 2$, or $|V(S_2)| = 3$ and $S_2$ contains a vertex of girth at least four. In either case, $(G_2, L'', S_2, A \cap V(G_2))$ is an unexceptional canvas and, by the minimality of $K$, admits an $L''$-colouring $\phi$.  Note that $\phi(u) \neq \phi_1(w)$, since $\phi$ is an $L''$-colouring. If $\phi(u) = \phi_1(u)$, then $\phi \cup \phi_1$ forms an $L$-colouring of $G$. Otherwise, $\phi(u) = \phi_2(u)$ and so $\phi \cup \phi_2$ forms an $L$-colouring of $G$. In either case, this is a contradiction.

We may therefore assume that $|L(u)| \leq 3$. For each $i \in \{1,2\}$, since $K[G_i]$ is a subcanvas of an unexceptional canvas, by Observation \ref{subcanvas} $K[G_i]$ admits an $L$-colouring $\varphi_i$. Let $L_{i}$ be the list assignment for $G_{i}$ obtained by setting $L_i(v) = L(v)$ for all $v \in V(G_i) -u$, and $L_i(u) = \{\varphi_{3-i}(u)\}$. For each $i \in \{1,2\}$, let $S_i'$ be the path obtained from $S_i$ by adding the edge $uw$, and let $K_i = (G_i, L_i, S_i', A \cap V(G_i))$. Note that both $K_1$ and $K_2$ are canvases. Suppose there exists some $i \in \{1,2\}$ such that $K_{i}$ is unexceptional. Then $G_i$ admits an $L_i$-colouring $\psi_i$, and $\psi_i \cup \varphi_{3-i}$ together form an $L$-colouring of $G$ \textemdash a contradiction. Note that since $|V(S_1)| \leq |V(S_2)|$, it follows that $|V(S_1')| \leq 3$ and so that $K_1$ is not an exceptional canvas of type (i) or (ii). 

We may therefore assume that $K_1$ is an exceptional canvas of type (iii). If $K_2$ is also an exceptional canvas of type (iii), then since $|L(u)| \leq 3$, we have that $K$ is an exceptional canvas of type (iii) \textemdash a contradiction. If $K_2$ is an exceptional canvas of type (i), then since $K_1$ is an exceptional canvas of type (iii) we have that $K$ is an exceptional canvas of type (ii) \textemdash a contradiction. Thus we may assume that $K_2$ is an exceptional canvas of type (ii) and thus contains a generalized wheel subgraph $W$ with principal path $u_1u_2u_3$, where $u_2$ and one vertex in $\{u_1,u_3\}$ are contained in $S_2'$, and the remaining vertex in $\{u_1,u_3\}$ forms a chord with $u_2$ in the outer cycle of $G_2$. Note that since $K_2$ is an exceptional canvas of type (ii), we have that $|V(S_2)| = 3$. Recall that every vertex in a generalized wheel has girth three. Suppose $uw$ is not a subpath of $u_1u_2u_3$. Then every vertex in $V(S_2) \setminus \{w\}$ has girth three. Since $K_1$ is an exceptional canvas of type (iii), every vertex in $V(S_1)$ has girth three. But then $S$ is not an acceptable path, as it contains four vertices of girth three. Thus we may assume that $uw$  is a subpath of $u_1u_2u_3$. But then $K$ is an exceptional canvas of type (ii), a contradiction.
\end{proof}



\subsection{Separating Cycle Lemmas}\label{separatingcyclelemmas}
In Lemmas \ref{sep3cycle}-\ref{sep6cycle}, we describe a set of reducible configurations for $K$: together, Lemmas \ref{sep3cycle} and \ref{sep4cycle} show that $G$ does not contain a cycle of length at most four with vertices in its interior. Lemmas \ref{sep5cycle} and \ref{sep6cycle} show that every cycle of length at most six with vertices in its interior is composed of vertices of relatively low girth. 

\begin{lemma}\label{sep3cycle}
If $T$ is a triangle in $G$, then $\Int(T) = \emptyset$. 
\end{lemma}
\begin{proof} 
Suppose not, and let $T = u_1u_2u_3u_1$ be a counterexample.  By the minimality of $K$, we have that $G-\Int(T)$ admits an $L$-colouring $\phi$. Let $G'$ be the graph obtained from $\Int[T]$ by deleting $u_3$. Let $C'$ be the outer cycle of $G'$. Let $L'$ be the list assignment obtained from $L$ by removing the colour $\phi(u_3)$ from the lists of all neighbours of $u_3$ in $\Int(C)$. Finally, let $L'(u_i) = \{\phi(u_i)\}$ for $i \in \{1,2\}$. Note that every vertex $v \in V(C') \setminus \{u_1, u_2\}$ with $\g(v) \in  \{3, 4\}$ has $|L'(v)| \geq 3$, since $|L(v)| \geq 4$. Let $A'$ be the set of vertices $v \in V(C')$ with $|L'(v)| = 2$. By the above, every vertex in $A'$ has girth at least five. Moreover, $A'$ is an independent set, since every vertex in $A'$ is adjacent to $u_3$ in $G$.  Thus since $u_1u_2$ is an acceptable path for $G'$ and contains only two vertices, it follows that $(G', L', u_1u_2, A')$ is an unexceptional canvas. Since $|V(G')| < |V(G)|$, by the minimality of $K$ we have that $G'$ has an $L'$-colouring $\phi'$. But $\phi \cup \phi'$ is an $L$-colouring of $G$, a contradiction.
\end{proof}

Our colouring arguments will be more straightforward if we assume that for every vertex $v \in V(C) \setminus (A \cup V(S))$, we have $|L(v)| = 3$. Unfortunately, deleting extra colours from lists in $V(C) \setminus (A \cup V(S))$ might result in the creation of an exceptional canvas of type (ii) or (iii). We show below this does not occur.

\begin{obs}\label{listsizes}
Every vertex in $S$ has a list of size one. Every vertex in $A$ has a list of size two. Every vertex in $V(C) \setminus (V(S)\cup A)$ has a list of size three. 
\end{obs}
\begin{proof}
Suppose not, and let $v \in V(G)$ be a counterexample. Since $K$ was chosen to minimize $\sum_{u \in V(G)} |L(u)|$, it follows that $|L(u)| = 1$ for every $u \in V(S)$. Thus $v \not \in V(S)$, as $v$ is a counterexample to the observation. Moreover, by definition of canvas every vertex in $A$ has exactly two colours in its list: thus $v \not \in A$. Since $v \in V(C) \setminus (V(S) \cup A)$ and $v$ is a counterexample to the observation, it follows that $|L(v)| \geq 4$. Since $K$ was chosen to minimize $\sum_{v \in V(G)} |L(v)|$, it follows that the canvas $K'$ obtained from $K$ by deleting a colour from the list of $v$ is an exceptional canvas. Since $C$ is chordless, $K'$ is not an exceptional canvas of type (ii). We claim moreover that $K'$ is not an exceptional canvas of type (i): this too follows from the fact that $C$ is chordless, and so if $K'$ is an exceptional canvas of type (i), then $C$ contains only the path $S$ together with a single vertex in $A$ adjacent to both endpoints of $S$, contradicting that $v \not \in V(S) \cup A$. Thus $K'$ is an exceptional canvas of type (iii), and so contains a subgraph $W$ that is a generalized wheel with principal path $S$ such that the vertices in the outer cycle of $W$ are in the outer face boundary of $G$ and have lists of size at most 3. Since $C$ is chordless by Lemma \ref{chordless}, $W$ is a wheel and moreover the outer cycle of $W$ is the outer cycle of $G$. Since every triangle in $G$ has no vertices in its interior by Lemma \ref{sep3cycle}, it follows that $W = G$.  Since $v \in V(C)$ and $W$ is a wheel, we have further that $\deg(v) = 3$. By the minimality of $K$, $G-v$ admits an $L$-colouring $\phi$. But then since $|L(v)| \geq 4$ and $\deg(v) = 3$, it follows that $\phi$ extends to an $L$-colouring of $G$ by choosing $\phi(v) \in L(v) \setminus \{\phi(u): u \in N(v)\}$, a contradiction.
\end{proof}
This observation allows us more easily to argue explicitly about the list assignment $L$ in Subsection \ref{towardsadeletablepath} and Section 3.

In a similar vein as Lemma \ref{sep3cycle}, we have the following. 
\begin{lemma}\label{sep4cycle}
If $T$ is a 4-cycle in $G$, then $\Int(T) = \emptyset$. 
\end{lemma}
\begin{proof}
Suppose not, and let $T= u_1u_2u_3u_4u_1$ be a counterexample. By the minimality of $K$, we have that $G-\Int(T)$ admits an $L$-colouring $\phi$. Let $G'$ be the graph obtained from $\Int[T]$ by deleting $u_3$ and $u_4$, and let $C'$ be the boundary walk of the outer face of $G'$.   Let $L'$ be the list assignment obtained from $L$ by removing the colour $\phi(u_i)$ from the lists of all neighbours of $u_i$ in $\Int(T)$, for $i\in \{3,4\}$. Finally, let $L'(u_j) = \{\phi(u_j)\}$ for $j\in \{1,2\}$. Note that every vertex $v \in V(C')\setminus \{u_1, u_2\}$ with $\g(v) = 3$ has $|L'(v)| \geq 3$, since $|L(v)| \geq 5$.  Moreover, every vertex $v$ in $V(C') \setminus \{u_1, u_2\}$ with $\g(v) = 4$ has $|L'(v)| \geq 3$, since $|L(v)| \geq 4$ and every such vertex $v$ is adjacent to at most one of $u_3$ and $u_4$ in $G$. Finally, let $A'$ be the set of vertices $v \in V(C')\setminus \{u_1, u_2\}$ with lists of size at most two. By the above, $\g(v) \geq 5$ for every $v \in A'$. It follows that very vertex in $A'$ is adjacent to exactly one of $u_3$ and $u_4$ in $G$, and moreover that $A'$ forms an independent set. Thus $|L'(v)| = 2$ for all $v \in A'$, since $|L(v)| \geq 3$ for all $v \in A'$. It follows that  $K' = (G', L', u_1u_2, A')$ is a canvas. Since $u_1u_2$ contains only two vertices, $u_1u_2$ is an acceptable path and moreover $K'$ is unexceptional. By the minimality of $K$ we have that $G'$ has an $L'$-colouring $\phi'$. By construction, $\phi \cup \phi'$ is an $L$-colouring of $G$, a contradiction.
\end{proof}

Note that by Theorems \ref{sep3cycle} and \ref{sep4cycle}, we have immediately that $S$ is not a cycle, and that $V(C) \setminus V(S) \neq \emptyset$. Thus $S$ is a path; for the remainder of the paper, let $S = v_1 v_2 \dots v_k$, and let $C = v_1 \dots v_k v_{k+1} \dots v_qv_1$. As noted in Subsection 2.2, we may assume without loss of generality that $k \geq 2$. Below, we show that in fact $k \geq 3$.  
\begin{obs}\label{kgeq3}
We may assume that $S$ contains at least three vertices.
\end{obs}
\begin{proof}
Suppose not. Thus $|V(S)|=2$. First suppose $|V(C)| \leq 4$. Then every vertex in $V(C)$ has girth at most four, and so $A = \emptyset$. By definition of canvas, we have that $|L(v)| \geq 3$ for all $v \in V(C) \setminus V(S)$. Since $C$ is chordless by Lemma \ref{chordless}, it then follows that $G[V(C)]$ has an $L$-colouring $\varphi$. Since $V(\Int(C)) = \emptyset$ by Lemmas \ref{sep3cycle} and \ref{sep4cycle} with $T = C$, it follows further that $G = C$ and so that $\varphi$ is an $L$-colouring of $G$, a contradiction.

Thus we may assume that $|V(C)| \geq 5$. Let $c$ be a colour in $L(v_3) \setminus L(v_2)$. Let $L'$ be a list assignment for $G$ defined by $L(v_3) = \{c\}$ and $L'(v) = L(v)$ for all $v \in V(G) \setminus \{v_3\}$. Since $S$ has only two vertices by assumption, $S+v_2v_3$ is an acceptable path. Moreover, $K' = (G, L', S+v_2v_3, A)$ is a canvas. Since $K$ was chosen to minimize $\sum_{v \in V(G)} |L(v)|$, it follows that $K'$ is an exceptional canvas; and since $|V(S+v_2v_3)| = 3$, $K'$ is an exceptional canvas of type (iii), and thus contains a subgraph $W$ that is a generalized wheel with principal path $S + v_2v_3$ such that the vertices on the outer cycle of $W$ are on the outer cycle of $G$ and have lists of size at most three under $L'$. Since $C$ is chordless by Lemma \ref{chordless} and $|V(C)| \geq 5$, it follows that $W$ is a wheel and moreover that the outer cycle of $W$ is the outer cycle of $G$. Since every triangle in $G$ has no vertices in its interior by Lemma \ref{sep3cycle}, we have that $W = G$. Let $w$ be the unique vertex not on the outer face boundary of $G$. Note that $N(v_3) = \{v_2, v_4, w\}$.

Since $v_3$ is in a wheel, $\g(v_3) = 3$ and so $|L(v_3)| = 3$. Let $X \subseteq L(v_3) \setminus L(v_2)$ be a set of size two, and let $L'$ be a list assignment for $G-v_3$ obtained by setting $L'(v) = L(v)$ for all $v \in V(G-v_3) \setminus \{w\}$ and $L'(w) = L(w) \setminus X$. Note that since $w \in V(\Int(C))$ and $\g(w) = 3$, it follows that $|L(w)| \geq 5$ and so that $|L'(w)| \geq 3$. Thus $K'' = (G-v_3, L', S, A)$ is a canvas. Since $|V(S)| = 2$, it is moreover unexceptional. By the minimality of $K$, we have that $K''$ admits an $L'$-colouring $\phi$. But $\phi$ extends to an $L$-colouring of $G$ by setting $\phi(v_3) \in X \setminus \{\phi(v_4)\}$, a contradiction.
\end{proof}

The following two lemmas show that only certain types of 5- and 6-cycles with vertices in their interior exist in $G$.

\begin{lemma}\label{sep5cycle}
If $P$ is a 5-cycle in $G$ and $\Int(P)$ contains a vertex, then every vertex in $V(P)$ has girth three.
\end{lemma}
\begin{proof}
Suppose not, and let $P \subseteq G$ be a counterexample with $P = u_1u_2u_3u_4u_5u_1$. Suppose without loss of generality that $\g(u_3) \geq 4$. By the minimality of $K$, we have that $G-\Int(P)$ admits an $L$-colouring $\phi$. Let $G'$ be the graph obtained from $\Int[P]$ by deleting the vertices $u_4$ and $u_5$, and let $C'$ be the outer face boundary of $G'$.  Let $L'$ be the list assignment obtained from $L$ by removing the colour $\phi(u_i)$ from the lists of all neighbours of $u_i$ in $\Int(P)$ for $i \in \{4,5\}$, and setting $L'(u_j) = \{\phi(u_j)\}$ for $j \in \{4,5\}$. Note that every vertex $v \in V(C')\setminus \{u_1, u_2, u_3\}$ with $\g(v) \leq 4$ has $|L'(v)| \geq 3$, since vertices in $\Int(C)$ of girth three have $|L(v)| \geq 5$, and vertices in $\Int(C)$ of girth four have $|L(v)| \geq 4$ and are adjacent to at most one of $u_4$ and $u_5$.   Finally, let $A'$ be the set of vertices $v \in V(C')\setminus \{u_1, u_2, u_3\}$ with lists of size at most two under $L'$. By the above, $\g(v) \geq 5$ for every $v \in A'$. It follows that that $A'$ forms an independent set, and moreover that each vertex in $A'$ is adjacent to at most one of $u_4$ and $u_5$.  Thus $|L'(v)| = 2$ for each $v \in A'$. Note that $u_1u_2u_3$ is an acceptable path for $G'$, as it has only three vertices. By the minimality of $K$, we have that $(G', L', u_1u_2u_3, A')$ is a canvas.  Since $\g(u_3) \geq 4$, it is unexceptional and thus $G'$ has an $L'$-colouring $\phi'$. By construction, $\phi \cup \phi'$ is an $L$-colouring of $G$, a contradiction.
\end{proof}

We end this subsection with the following lemma, characterizing the separating cycles of length six. 
\begin{lemma}\label{sep6cycle}
If $H$ is a 6-cycle in $G$ and $\Int(H)$ contains a vertex, then $V(H)$ does not contain a vertex of girth at least five.
\end{lemma}
\begin{proof}
Suppose not, and let $H \subseteq G$ be a counterexample with $H = u_1u_2 \cdots u_6u_1$. Suppose without loss of generality that $\g(u_3) \geq 5$. By the minimality of $K$, we have that $G-\Int(H)$ admits an $L$-colouring $\phi$. The argument proceeds in a way similar to that of Lemmas \ref{sep3cycle}-\ref{sep5cycle}: we aim to delete a path of vertices in $H$ and argue about the resulting graph. Which vertices we delete from $H$ will depend on the structure of $\Int(H)$. 

First, suppose that there is no vertex in $V(\Int(H))$ adjacent to $u_4$, $u_5$, and $u_6$; that there is no vertex $w$ in $V(\Int(H))$ with $\g(w) = 4$ such that $w$ is adjacent to both $u_4$ and $u_6$; and that there do not exist vertices $w_1, w_2$ in $V(\Int(H))$ such that $\g(w_1) \geq 5$ and $\g(w_2) \geq 5$ and $\{w_1u_4, w_2u_6, w_1w_2\} \subseteq E(G)$. In this case, let $G'$ be the graph obtained from $\Int[H]$ by deleting vertices $u_4$, $u_5$ and $u_6$, and let $C'$ be the outer cycle of $G'$. Let $L'$ be the list assignment obtained from $L$ by removing the colour $\phi(u_i)$ from the lists of all neighbours of $u_i$ in $\Int(H)$ for $i\in \{4,5,6\}$. Let $L'(u_i) = \{\phi(u_i)\}$ for $i \in \{1,2,3\}$. Note that every vertex $v \in V(C') \setminus \{u_1, u_2, u_3\}$ with $\g(v) = 3$ has $|L'(v)| \geq 3$, since vertices in $\Int(C)$ of girth three have $|L(v)| \geq 5$ and are adjacent to at most two of $u_4, u_5$, and $u_6$ by assumption. Moreover, every vertex $v \in V(C') \setminus \{u_1, u_2, u_3\}$ with $\g(v) = 4$ has $|L'(v)| \geq 3$ since vertices in $\Int(C)$ of girth four have $|L(v)| \geq 4$ and are adjacent to at most one of $u_4$, $u_5$, and $u_6$ by assumption. Finally, let $A'$ be the set of vertices $v \in V(C')\setminus \{u_1, u_2, u_3\}$ with lists of size at most two under $L'$. By the above, $\g(v) \geq 5$ for every $v \in A'$. It follows that every vertex in $A'$ is adjacent to exactly one of $u_4, u_5,$ and $u_6$, and so that $|L'(v)| = 2$ for each $v \in A'$ (since $|L(v)| \geq 3$). Finally, we note that $A'$ forms an independent set by assumption. Thus $(G', L', u_1u_2u_3, A')$ is a canvas. Since $u_1u_2u_3$ contains only three vertices, it is an acceptable path; and since $\g(u_3) \geq 5$, we have that $(G', L', u_1u_2u_3, A')$ is unexceptional. By the minimality of $K$, we have further that $G'$admits an $L'$-colouring $\phi'$. By construction, $\phi \cup \phi'$ is an $L$-colouring of $G$, a contradiction. 

Thus we may assume that there is a vertex in $V(\Int(H))$ adjacent to $u_4$, $u_5$, and $u_6$; or that there is a vertex $w$ in $V(\Int(H))$ with $\g(w) = 4$ such that $w$ is adjacent to both $u_4$ and $u_6$; or finally that there exist vertices $w_1, w_2$ in $V(\Int(H))$ such that $\g(w_1) \geq 5$ and $\g(w_2) \geq 5$ and $\{w_1u_4, w_2u_6, w_1w_2\} \subseteq E(G)$.  We break into cases depending on which of these occur. 

\vskip 4mm
\noindent
\textbf{Case 1: Either there is a vertex $w$ in $V(\Int(H))$ with $\g(v) = 4$ such that $w$ is adjacent to both $u_4$ and $u_6$, or there exist vertices $w_1, w_2$ in $V(\Int(H))$ such that $\g(w_1) \geq 5$ and $\g(w_2) \geq 5$ and $\{w_1u_4, w_2u_6, w_1w_2\} \subseteq E(G)$.} Note that by the planarity of $G$ and by Lemmas \ref{sep4cycle}-\ref{sep5cycle}, it follows that there does not exist a vertex in $V(\Int(H))$ adjacent to $u_5$.  In this case, let $G'$ be the graph obtained from $\Int[H]$ by deleting $u_1, u_5,$ and $u_6$, and let $C'$ be the outer cycle of $G'$. Let $L'$ be the list assignment obtained from $L$ by removing the colour $\phi(u_i)$ from the lists of all neighbours of $u_i$ in $\Int(H)$ for $i\in \{1, 5, 6\}$. Let $L'(u_i) = \{\phi(u_i)\}$ for $i\in \{2,3,4\}$. Note that since no vertex in $\Int(H)$ is adjacent to $u_5$, there does not exist a vertex $w$ in $\Int(H)$ adjacent to $u_1, u_5,$ and $u_6$.  Similarly,  there does not exist a vertex $w$ in $\Int(H)$ with $\g(w) = 4$ such that $w$ is adjacent to $u_1$ and $u_5$. Finally, there do not exist vertices $w_1, w_2$ in $\Int(H)$ such that $\g(w_1) \geq 5$ and $\g(w_2) \geq 5$ and $\{w_1u_1, w_2u_5, w_1w_2\} \subseteq E(G)$. Let $A'$ be the set of vertices $v \in V(C')$ with $\g(v) \geq 5$ and $|L'(v)| = 2$.  It follows as in the previous cases that $A'$ is an independent set, and moreover that $(G', L', u_2u_3u_4, A')$ is a canvas. Since $\g(u_3) \geq 5$, it is unexceptional. By the minimality of $K$, we have that $G'$ admits an $L'$-colouring $\phi$. By construction, $\phi \cup \phi'$ is an $L$-colouring of $G$, a contradiction.

\vskip 4mm
\noindent
\textbf{Case 2: there is a vertex $w$ in $\Int(H)$ adjacent to $u_4$, $u_5$, and $u_6$.}  First suppose $w$ is not adjacent to $u_1$. In this case, as in Case 1 we delete $u_1, u_5$, and $u_6$. The argument is nearly identical as that for Case 1, with one caveat: $w$ is adjacent to $u_5$. However, by Lemma \ref{sep3cycle}, it is the only vertex in $\Int(H)$ adjacent to $u_5$, and since $\g(w) = 3$, we have that $|L(w)| \geq 5$. This ensures $|L'(w)| \geq 3$ in the argument for Case 1.  We may thus assume $w$ is adjacent to $u_1$. In this case, by Lemmas \ref{sep3cycle} and \ref{sep5cycle} (noting $\g(u_3) \geq 5$), it follows that $w$ is the only vertex in the interior of $H$. By the minimality of $K$, we have that $G-w$ admits an $L$-colouring $\phi$. Since $\g(u_3) \geq 5$, we have that $w$ is adjacent to only $u_1, u_6, u_5,$ and $u_4$  on $H$. As $|L(w)| \geq 5$, it follows that $\phi$ extends to an $L$-colouring of $G$, a contradiction.
\end{proof}

\subsection{Separating Path Lemmas}\label{seppathlemmas}
The proofs of many of our lemmas take the following basic shape: we colour and delete vertices in $V(C)\setminus V(S)$, modifying lists where appropriate, and argue about the structure of the resulting canvas $(G', L', S', A')$. The following lemma shows that in doing this, as long as the set of precoloured vertices does not change (i.e. as long as $S' = S$), we do not create an exceptional canvas of type (i). Moreover, it shows that there are no edges between vertices in $A'$ and vertices in $A$.  We will use this lemma extensively throughout the remainder of the paper.

\begin{lemma}\label{intg4s}
If $v$ is a vertex in $V(\Int(C))$ that is adjacent to two vertices $u,w$ in $(V(C) \setminus V(S)) \cup \{v_1, v_k\}$, then $\g(x) = 3$ for every $x \in \{u,v,w\}$.
\end{lemma}
\begin{proof}
Suppose not. The path $uvw$ separates $G$ into two graphs $G_1$ and $G_2$ (note that since $\g(x) \geq 4$ for at least one $x \in \{u,v,w\}$, it follows that for each $i \in \{1,2\}$, we have that $V(G_i) \setminus V(G_{3-i}) \neq \emptyset$). Since neither $u$ nor $w$ is an internal vertex of the path $S$, we may assume without loss of generality that $S \subseteq G_1$.  

Note that since $K$ is unexceptional, it follows from Observation \ref{subcanvas} that $K[G_1]$ is an unexceptional canvas. By the minimality of $K$, we have that $G_1$ admits an $L$-colouring $\phi$. Let $L'$ be a list assignment for $G_2$ defined by $L'(x) = \{\phi(x)\}$ for $x \in \{u, v, w\}$, and $L'(v) = L(v)$ for $x \in V(G_2) \setminus \{u, v, w\}$. Since $uvw$ has only three vertices, it is an acceptable path for $G_2$. Note that $(G_2, L', uvw, A \cap V(G_2))$ is a canvas, and since $\g(x) \geq 4$ for at least one vertex $x \in \{u,v,w\}$, it is moreover unexceptional. Since $|V(G_2)| < |V(G)|$ and $K$ is a minimum counterexample to Theorem \ref{345colouring},  $G_2$ admits an $L'$-colouring $\phi'$. But then $\phi \cup \phi'$ is an $L$-colouring of $G$, a contradiction. 
\end{proof}

In a similar spirit, we have the following lemma which partially describes the structure of $G$ surrounding separating paths of length three whose inner vertices have girth at least five.
\begin{lemma}\label{3-5-5-3}
If $v$ and $w$ are vertices in $V(\Int(C))$ such that:
\begin{itemize}
    \item $\g(v) \geq 5$ and $\g(w) \geq 5$, 
    \item $vw \in E(G)$, 
    \item $v$ is adjacent to a vertex $v' \in (V(C) \setminus V(S)) \cup \{v_1, v_k\}$, and
    \item $w$ is adjacent to a vertex $w' \in (V(C) \setminus V(S)) \cup \{v_1, v_k\}$,
\end{itemize}
then there exists a vertex $u \in A$ that is adjacent to both $v'$ and $w'$.
\end{lemma}
\begin{proof}
The proof is nearly identical to that of Lemma \ref{intg4s}: suppose not. Note that $w' \neq v'$ and $w'v' \not \in E(G)$ since $\g(v) \geq 5$. Thus the path $w'wvv'$ separates $G$ into two graphs $G_1$ and $G_2$. Since neither $w'$ nor $v'$ is an internal vertex of the path $S$, we may assume without loss of generality that $S \subseteq G_1$. 

Since $K$ is unexceptional, by Observation \ref{subcanvas} so too is $K[G_1]$.  By the minimality of $K$,  we have that $G_1$ admits an $L$-colouring $\phi$. Let $L'$ be a list assignment for $G_2$ defined by $L'(x) = \{\phi(x)\}$ for $x \in \{w',w,v,v'\}$, and $L'(v) = L(v)$ for $x \in V(G_2) \setminus \{w',w,v,v'\}$. Note that since $\g(x) \geq 5$ for both $x \in \{v,w\}$, it follows that $w'wvv'$ is an acceptable path for $G_2$. Note that $K_2 =(G_2, L', w'wvv', A \cap V(G_2))$ is a canvas. If $K_2$ is an exceptional canvas of type (i), then there is a vertex $u \in A$ adjacent to both $v'$ and $w'$, a contradiction. Moreover, $K_2$ is not an exceptional canvas of type (ii) since both $w$ and $v$ have girth at least five and every vertex in a generalized wheel has girth three. Finally, $K_2$ is not an exceptional canvas of type (iii) since its precoloured path has four vertices. Since $|V(G_2)| < |V(G)|$ and $K$ is a minimum counterexample to Theorem \ref{345colouring},  $G_2$ admits an $L'$-colouring $\phi'$. But then $\phi \cup \phi'$ is an $L$-colouring of $G$, a contradiction.
\end{proof}

\subsection{Broken Wheel Lemmas}\label{brokenwheellemmas}
The next lemma restricts the set of possible generalized wheels contained in $G$.

\begin{lemma}\label{nogenwheels}
Suppose $u$ is a vertex in $\Int(C)$ that is adjacent to two vertices $v_i$ and $v_j$ on $C$, where $\{i,j\} \cap \{2, \dots, k-1\} = \emptyset$ (so that neither $v_i$ nor $v_j$ is an internal vertex of the path $S$). Let $Q$ be the path in $C$ with endpoints $v_i$ and $v_j$ containing no edges of $S$. Then $G$ contains a subgraph $W$ which is a broken wheel with principal path $v_iuv_j$ such that $Q$ is in the outer face boundary of $W$.
\end{lemma}
\begin{proof}
Suppose not. Note that by Lemma \ref{intg4s}, $v_i$, $u$, and $v_j$ all have girth three. The path $v_iuv_j$ separates $G$ into two subgraphs $G_1$ and $G_2$. (By Observation \ref{kgeq3}, $S$ contains at least three vertices and so it follows that $V(G_i) \setminus V(G_{3-i}) \neq \emptyset $ for each $i \in \{1,2\}$.) Since neither $v_i$ nor $v_j$ is an internal vertex of the path $S$, we may assume without loss of generality that $S \subseteq G_1$.  By Observation \ref{subcanvas},  $K[G_1]$ is an unexceptional canvas. By the minimality of $K$, it follows that $G_1$ admits an $L$-colouring $\phi$. Let $L'$ be a list assignment for $G_2$ obtained from $L$ by setting $L'(v) = \{\phi(v)\}$ for $v \in \{v_i, u, v_j\}$, and $L'(v) = L(v)$ otherwise. Then $K_2 = (G_2, L', v_iuv_j, A \cap V(G_2))$ is a canvas. Note that since $v_iuv_j$ has only three vertices, it is an acceptable path. If $K_2$ is unexceptional, then $G_2$ admits an $L'$-colouring $\phi'$; but then $\phi \cup \phi'$ forms an $L$-colouring of $G$, a contradiction. Thus $K_2$ is exceptional, and since $v_iuv_j$ has only three vertices, $K_2$ is an exceptional canvas of type (iii).

 By assumption, $G_2$ does not contain a subgraph $W$ that is a broken wheel with principal path $v_iuv_j$ such that the vertices on the outer cycle of $W$ are on the outer cycle of $G_2$. Since $K_2$ is an exceptional canvas of type (iii), it follows that $G_2$ contains a subgraph $W$ that is a generalized wheel but not a broken wheel such that the vertices on the outer cycle of $W$ are on the outer cycle of $G_2$. By Lemma \ref{chordless}, every chord of the outer cycle of $G_2$ has $u$ as an endpoint. Thus all edges on the outer cycle of $W$ are on the outer cycle of $G_2$. By Lemma \ref{sep3cycle}, every triangle in $G_2$ has no vertices in its interior. It follows that $W = G_2$, and so that $G_2$ is a near-triangulation. By Theorem \ref{thomassenbadcol}, there is at most one colouring of $v_iuv_j$, say $\varphi$, that does not extend to a colouring of $G_2$. Let $L^\star(u) = L(u) \setminus \varphi(u)$, and $L^\star(v) = L(v)$ for all $v \in V(G) \setminus \{u\}$. Note that since $\g(u) = 3$ and $u \in V(\Int(C))$, we have that $|L(u)| \geq 5$. It follows that $|L^\star(u)| \geq 4$. 

Since $K[G_1]$ is unexceptional and $|L^\star(u)| \geq 4$, it follows that $(G_1, L^\star, S, A \cap V(G_1))$ is an unexceptional canvas. By the minimality of $K$, we have that $G_1$ admits an $L^\star$-colouring $\varphi^\star$. Let $L^{\star \star}$ be the list assignment for $G_2$ obtained by setting $L^{\star \star}(v) = \{\varphi^\star(v)\}$ for all $v \in \{v_i, u, v_j\}$ and $L^{\star \star}(v) = L(v)$ otherwise. By Theorem \ref{thomassenbadcol} and the fact that $\varphi^\star(u) \neq \varphi(u)$, we have that $(G_2, L^{\star \star}, v_iuv_j, A \cap V(G_2))$ admits an $L^{\star \star}$-colouring $\varphi^{\star \star}$. As $\varphi^\star \cup \varphi^{\star \star}$ is an $L$-colouring of $G$, this is a contradiction.

\end{proof}

The following lemma provides some insight into the structure surrounding the broken wheels described by Lemma \ref{nogenwheels}. It will be useful in bounding the size of certain broken wheels in $G$.
\begin{lemma}\label{closeafan}
Suppose $W \subseteq G$ is a broken wheel with outer cycle $ww_1w_2 \dots w_tw$ and principal path $w_1ww_t$ such that $w \in V(\Int(C))$; $V(W)\setminus \{w\} \subseteq (V(C) \setminus V(S)) \cup \{v_1, v_k\}$; and $t \geq 3$. Let $1 \leq j \leq t-2$, and let $G'$ be the graph obtained from $G$ by identifying $w_j$ and $w_{j+2}$ to a new vertex $z$ and deleting $w_{j+1}$. Let $x \in V(G')\setminus \{z\}$. The following both hold. 
\begin{enumerate}
    \item If $\g_G(x) \geq 5$, then $\g_{G'}(x) \geq 5$.
    \item If $\g_G(x) = 4$, then $\g_{G'}(x) \geq 4 $.
\end{enumerate}
\end{lemma}
\begin{proof}
Suppose not. Then $\g_G(x) \geq 4$, and $\g_{G'}(x) < \g_G(x)$. Moreover, $x \not \in V(W)$, since every vertex in a broken wheel has girth three. By Lemma \ref{chordless}, the outer cycle $C$ of $G$ is chordless: thus all edges on the outer cycle of $W$ other than $ww_1$ and $ww_t$ are in $E(C)$. Moreover, every triangle in $G$ \textemdash and so in particular, every triangle in $W$ \textemdash has no vertex in its interior by Lemma \ref{sep3cycle}.  Let $H'$ be a smallest cycle in $G'$ containing $x$. Note that $H'$ is induced. Since $\g_{G'}(x) < \g_G(x)$, it follows that $z \in V(H')$. First suppose that $w \in V(H')$. Then since $z \in V(H')$ and $H'$ is an induced cycle it follows that $zw \in E(H')$. Let $w' \neq w$ be a neighbour of $z$ in $H'$. Since $z$ is the identification of $w_j$ and $w_{j+2}$, it follows that $G$ contains a cycle $P$ obtained from $H'$ by replacing the path $wzw'$ with one of $ww_jw'$ and $ww_{j+2}w'$. This is a contradiction, since then $|V(P)| = |V(H')|$, but by assumption $|V(H')| < \g_G(x)$. 

Thus we may assume that $w \not \in V(H')$. Since every triangle in $W$ has no vertex in its interior by Lemma \ref{sep3cycle}, it follows that $\{w_1, \dots, w_t\} \setminus \{w_j, w_{j+1}, w_{j+2}\} \subseteq V(H')$. Let $H$ be the cycle in $G$ obtained from $H'$ by replacing $z$ by the path $w_{j}w_{j+1}w_{j+2}$. Since every triangle in $W$ has no vertices in its interior and $w \not \in V(H)$, it follows that $w \in V(\Int(H))$. Note that $|V(H')| \geq 3$, and so $|V(H)|\geq 5$. If $|V(H)| = 5$, this contradicts Lemma \ref{sep5cycle} since $\g_G(x) \geq 4$ by assumption and $x \in V(H)$. If $|V(H)| \geq 7$, then $|V(H')|\geq 5$ and so $\g_{G'}(x) \geq 5$, a contradiction. Thus we may assume that $|V(H)| = 6$, and so that $|V(H')| = 4$. It follows that $\g_G(x) \geq 5$. But then this contradicts Lemma \ref{sep6cycle}, since $x \in V(H)$.

\end{proof}

We are now equipped to bound the size of certain broken wheels in $G$.

\begin{lemma}\label{nobigfans}
Suppose that $W \subseteq G$ is a broken wheel with outer cycle $ww_1w_2 \cdots w_tw$ and principal path $w_1ww_t$ such that $w \in V(\Int(C))$; $V(W) \setminus \{w\} \subseteq V(C) \setminus V(S)$; and $t \geq 3$. There does not exist an index $1 \leq j \leq t-2$ such that $L(w_j) = L(w_{j+2})$.
\end{lemma}
\begin{proof}
Suppose not. Let $j$ be an index with $1 \leq j \leq t-2$ and $L(w_j) = L(w_{j+2})$.  By Lemma \ref{chordless}, $C$ is chordless, and so every every edge on the outer cycle of $W$ other than $ww_1$ and $ww_t$ is in $E(C)$. Moreover, note that every triangle in $W$ has no vertices in its interior by Lemma \ref{sep3cycle}.  Let $G'$ be the graph obtained from $G$ by identifying $w_j$ and $w_{j+2}$ to a new vertex $z$ and deleting $w_{j+1}$. Note that $G'$ is planar, and inherits the embedding of $G$. Set $L(z) = L(w_j)$. Since $w_{j}$ is in a broken wheel in $G$,  $\g_G(w_j) = 3$. Since $w_j \not \in V(S)$, it follows from Observation \ref{listsizes} that $|L(w_j)| = |L(z)| = 3$. (Similarly, $|L(w_{j+1})| = |L(w_{j+2})| = 3$.) By Lemma \ref{closeafan}, we have that for all $x \in V(G') \setminus \{z\}$, if $\g_G(x) \geq 5$, then $\g_{G'}(x) \geq 5$; and that if $\g_G(x) = 4$, then $\g_{G'}(x) \geq 4 $.

It follows from this that $K' = (G', L, S, A)$ is a canvas. Moreover, it follows from Lemma \ref{closeafan} that $S$ is an acceptable path in $G'$. If $K'$ is unexceptional, then since $K$ is a minimum counterexample it follows that $G'$ admits an $L$-colouring $\phi$. This is a contradiction, since $\phi$ extends to an $L$-colouring of $G$ by setting $\phi(w_j) = \phi(w_{j+2}) = \phi(z)$, and choosing $\phi(w_{j+1}) \in L(w_{j+1}) \setminus \{\phi(w_j), \phi(w)\}$. (Since $\g(w_{j+1}) = 3$, it follows that $|L(w_{j+1})| = 3$ and so that $w_{j+1}$ receives a colour.)  

Thus $K'$ is an exceptional canvas. Since $C$ is chordless by Lemma \ref{chordless}, it follows that $K'$ is not an exceptional canvas of type (i) or (ii). Thus we may assume that $K'$ is an exceptional canvas of type (iii), and so that $G'$ contains a generalized wheel $W'$ with principal path $S$ such that the vertices on the outer cycle of $W'$ are on the outer cycle of $G'$. Again because $C$ is chordless it follows that the outer cycle of $W'$ is the outer cycle of $G'$, and that the outer cycle of $G'$ is also chordless. Since the outer cycle of every generalized wheel that is not a wheel or a triangle has a chord, we have that $W'$ is a wheel. It follows that every triangle in $G'$ corresponds to a triangle in $G$ (replacing $z$ by $w_j$ or $w_{j+2}$ where appropriate). By Lemma \ref{sep3cycle}, we have that $W' = G'$; and since $w \in V(\Int(C))$ and $G'$ is a wheel, it follows further that $w$ is the only vertex in $W'$ not in the outer cycle of $G'$. But then $G$ too is a wheel with principal path $S$, and since $|L(w_{j+1})| = 3$ by Observation \ref{listsizes}, it follows that $K$ is an exceptional canvas of type (iii) \textemdash a contradiction.
\end{proof}

To close this subsection, we give one last lemma: it restricts still further the set of broken wheel subgraphs in $G$. It is used in Section 3 to argue concisely that upon performing our main reductions (colouring and deleting a subset of the vertices of $G$, and adjusting lists where appropriate), what remains is not an exceptional canvas of type (iii).
\begin{lemma}\label{deg3forv4}
If $S$ has length two and contains only vertices of girth three and $|V(C)| \geq 5$, then $V(\Int(C))$ does not contain a vertex $w$ adjacent to $v_3$, $v_{4}$, and $v_{5}$.
\end{lemma}
\begin{proof}
Suppose not. By Lemma \ref{sep3cycle}, triangles in $G$ have no vertices in their interior, and so $\Int(wv_3v_{4}w) = \emptyset$ and $\Int(wv_{4}v_{5}w) = \emptyset$. Thus $N(v_4) = \{v_3, v_5, w\}$. Since $v_{4}$ has girth three, it is not contained in $A$. Thus by Observation \ref{listsizes}, we have that $|L(v_3)| = 1$ and $|L(v_{4})| = 3$. We claim $L(v_3) \subset L(v_{4})$. To see this, suppose not. Let $G^*$ be the graph obtained from $G$ by deleting $v_3$, and let $L^*$ be a list assignment for $G^*$ obtained by setting $L^*(v) = L(v)$ for all $v \in V(G) \setminus N_G(v_3)$, and $L^*(v) = L(v) \setminus L(v_3)$ for all $v \in N_G(v_3)$. Let $C^*$ be the graph whose vertex- and edge-set are precisely those of the outer face boundary of $G^*$. Let $A^*$ be the set of vertices with lists of size two under $L^*$. We now show that $K^* =(G^*, L^*, S-v_3, A^*)$ is a canvas. To see this, note that $|L^*(v_4)| = |L(v_4)|$ by assumption, and that every vertex $v \in V(C^*) \setminus V(C)$ satisfies $|L^*(v)| \geq |L(v)| -1$. Thus every vertex in $A^*\setminus A$ has $|L(v)| = 3$ and hence has girth at least $5$ in $G$. Thus $A^*\setminus A$ is an independent set. We claim moreover $A^*$ is an independent set. This follows from Lemma \ref{intg4s} and the fact that every vertex in $A^*\setminus A$ is adjacent to $v_3$. Finally, $S-v_3$ has only two vertices; hence $S-v_3$ is an acceptable path in $G^*$ and $K^*$ is unexceptional. By the minimality of $K$ there is an $L$-colouring $\varphi$ of $K^*$ which extends to an $L$-colouring of $K$ by setting $\varphi(v_3) \in L(v_3)$, a contradiction. This proves the claim.

Thus $L(v_3) \subset L(v_4)$. By assumption, $\g(v_4) = 3$ and so $|L(v_4)| = 3$. Without loss of generality we may assume that $L(v_3) = \{1\}$ and that $L(v_{4}) = \{1,2,3\}$. Let $L'$ be the list assignment obtained from $L$ by setting $L'(w) = L(w) \setminus \{2,3\}$, and $L'(v) = L(v)$ for all $v \in V(G)\setminus \{w\}$. Let $C'$ be the graph with vertex- and edge-set precisely those of the outer face boundary walk of $G-v_4$. (Thus $C'$ is the cycle obtained from $C$ by replacing the path $v_3v_4v_5$ by the path $v_3wv_5$.)  
Note that since $w \in V(\Int(C))$ and $\g(w) = 3$, it follows that $|L(w)| \geq 5$ and so that $|L'(w)| \geq 3$. Thus $K' = (G-v_{4}, L', S, A)$ is a canvas.

If $K'$ is unexceptional, then by the minimality of $K$ we have that $G-v_4$ admits an $L'$-colouring $\phi$ which extends to an $L$-colouring of $G$ by choosing $\phi(v_4) \in L(v_4) \setminus \{\phi(v_3), \phi(v_5)\}$, a contradiction.

Thus we may assume that $K'$ is an exceptional canvas. Since $S$ is a path of length two, we have further that $K'$ is an exceptional canvas of type (iii). Thus $G-v_4$ contains a subgraph $W$ that is a generalized wheel with principal path $S$ such that the vertices on the outer cycle of $W$ are on $C'$. Since $V(C')\setminus \{w\} \subset V(C)$ and $K$ is unexceptional, we have that $w \in V(W)$ and moreover that $w$ is in the outer cycle of $W$. Let $u$ be the neighbour of $w$ in the outer cycle $W$ with $u \neq v_3$. Note that $u \in V(C)$. Thus the path $v_3wu$ separates $G$ into two graphs $G_1$ and $G_2$ where without loss of generality $S \subseteq G_1$, and since $C$ is chordless by Lemma \ref{chordless}, the outer cycle of $G_1$ is the outer cycle of $W$. By Lemma \ref{sep3cycle}, every triangle in $W$ (and thus in $G_1$) has no vertex in its interior and so $G_1$ is a near-triangulation. By Lemma \ref{nogenwheels}, $w$ is adjacent to every vertex in the subpath of $C$ with endpoints $v_3$ and $u$ containing no edges of $S$ and hence by Lemma \ref{sep3cycle}, $G_2$ is a near-triangulation. Thus $G$ is a near-triangulation. Since $K$ is not exceptional, we have by Theorem \ref{thomassen3ext} that $G$ admits an $L$-colouring, a contradiction.

\end{proof}
\subsection{Towards a Deletable Path}\label{towardsadeletablepath}
In this subsection, we give several lemmas necessary for establishing the existence of our main reducible configuration, called a \emph{deletable path}.  We end this subsection with the precise definition of the path.  Recall that $S = v_1v_2 \dots v_k$, and that $C = v_1v_2 \cdots v_k \cdots v_qv_1$. Moreover, we have assumed that $G[V(S)]$ has an $L$-colouring. By Lemmas \ref{sep3cycle} and \ref{sep4cycle} and the fact that $k \leq 4$ by definition, it follows that $S$ is not a cycle: that is, $q \neq k$. Lemma \ref{extraverts1} shows that in fact $C$ contains several non-$S$ vertices: in particular, that $q \geq k+3$. In Lemma \ref{extraverts2}, we show that if $v_{k+1} \in A$, then in fact $q \geq k+4$. Lemma \ref{colours} partially describes the list assignment of $v_{k+1}, v_{k+2},$ and $v_{k+3}$. Corollary \ref{notthird} restricts $A \cap \{v_{k+1}, v_{k+2}, v_{k+3}, v_{k+4}\}$ and is used in showing that $G$ contains a deletable path, formally defined in Definition \ref{defdelpath}.

\begin{lemma}\label{extraverts1}
$q \geq k+3$.
\end{lemma}
\begin{proof}
Suppose not. As noted above, $k \neq q$. We claim that $G[V(C)]$ admits an $L$-colouring. To see this, suppose  that $V(C) \setminus V(S)$ contains only vertices with lists of size two. Then by definition $V(C) \setminus V(S) \subseteq A$ \textemdash and since $A$ is an independent set in $G$, it follows that $|V(S)| =4$, and $A$ contains only a single vertex that is adjacent to both $v_1$ and $v_4$. This is a contradiction, since $K$ is not an exceptional canvas of type (i). It follows that $V(C) \setminus V(S)$ contains a vertex with a list of size three. Since $C$ is chordless by Lemma \ref{chordless}, we have then that $G[V(C)]$ is $L$-colourable. Therefore since $G$ does not admit an $L$-colouring, $V(\Int(C)) \neq \emptyset$.

First assume $q = k+1$. Note that $k = 4$ since otherwise one of Lemmas \ref{sep3cycle} and \ref{sep4cycle} gives a contradiction. By Lemma \ref{sep5cycle}, every vertex in $C$ (and so in particular, every vertex in $V(S)$) has girth three. But then $S$ is not an acceptable path, a contradiction. 

Next, assume $q = k+2$. Similar to the previous case, if $k=1$ then $C$ contradicts Lemma \ref{sep3cycle}. If $k = 2$, then $C$ is a 4-cycle with a vertex in its interior, contradicting Lemma \ref{sep4cycle}. If $k=3$, then since $C$ is a 5-cycle with a vertex in its interior, by Lemma \ref{sep5cycle} every vertex in $S$ has girth three. Note that since $C$ is chordless by Lemma \ref{chordless}, $C$ admits an $L$-colouring. Let $\phi$ be an $L$-colouring of $C$, and let $G'$ be the graph obtained from $G$ by deleting $v_4$ and $v_5$. Let $C'$ be the graph whose vertex- and edge-set are precisely those of the outer face boundary of $G'$. Let $L'$ be the list assignment for $G'$ obtained from $L$ by setting $L'(v) = \{\phi(v)\}$ for every $v \in V(S)$ and $L'(v) = L(v) \setminus \{\phi(x): x \in N_G(v) \cap \{v_4, v_5\}\}$ for all $v \in V(G') \setminus V(S)$.

Let $A'$ be the set of vertices in $V(C') \setminus V(S)$ with lists of size at most two under $L'$. Note that every vertex $v \in V(C') \setminus V(S)$ of girth at most four has $|L'(v)| \geq 3$, since vertices in $\Int(C)$ of girth three had lists of size at least five and lost at most two colours, and vertices in $\Int(C)$ of girth four had lists of size at least four and lost at most one colour. Thus the vertices in $A'$ all have girth at least five, and so $A'$ is an independent set, as every vertex in $A'$ is adjacent to either $v_4$ or $v_5$ in $G$. Note furthermore that as $A'$ contains only vertices of girth at least five, no vertex in $A'$ is adjacent to both $v_4$ and $v_5$ in $G$: thus $|L'(v)| = 2$ for all $v \in A'$. Furthermore, $(G', L', S, A')$ does not admit an $L'$-colouring $\phi'$, as otherwise $\phi \cup \phi'$ forms an $L$-colouring of $G$, a contradiction. By the minimality of $K$, it follows that $(G', L', S, A')$ is an exceptional canvas of type (iii): in other words, $G'$ contains a generalized wheel $W$ with principal path $S$ such that the vertices on the outer cycle of $W$ are on the outer face boundary of $G'$, and no vertex in the outer cycle of $W$ has a list of size more than three. By the definition of generalized wheel, all vertices in the outer face boundary of $W$ have girth three. It follows that all vertices in the outer cycle of $W$ that are not in $V(C)$ are in $\Int(C)$, have girth three, and (as they have lists of size at most three under $L'$) are adjacent to both $v_4$ and $v_5$. Since triangles in $G$ have no vertices in their interiors by Lemma \ref{sep3cycle}, it follows that there exists exactly one vertex $u \in V(\Int(C))$ that is adjacent to both $v_4$ and $v_5$, and therefore only $u$ is in the outer cycle of $W$ and not in $V(C)$.  Thus $u$ is adjacent $v_1$, $v_3$, $v_4$ and $v_5$ in $G$. Since 4-cycles in $G$ have no vertices in their interiors by Lemma \ref{sep4cycle} and $\g(v_2) = 3$ since $v_2$ is in a generalized wheel $W$, it follows that $u$ is also adjacent to $v_2$. But then $G$ is a wheel, and $K$ is an exceptional canvas of type (iii) \textemdash a contradiction.

Thus we may assume $k = 4$, and so that $q = 6$. In this case, let $\phi$ be an $L$-colouring of $C$. (Note that $\phi$ exists, since $C$ is chordless by Lemma \ref{chordless} and $V(C)\setminus V(P)$ contains at least one vertex $v$ with $|L(v)| \geq 3$ since $A$ is an independent set.) Let $G'$ be the graph obtained from $G$ by deleting $v_6$ and $v_5$, and let $C'$ be the outer cycle of $G'$. Let $L'$ be the list assignment for $G'$ obtained from $L$ by setting $L'(v) = \{\phi(v)\}$ for every $v \in V(S)$ and $L'(v) = L(v) \setminus \{\phi(x): x \in N_G(v) \cap \{v_5, v_6\}\}$ for all $v \in V(G') \setminus V(S)$.

Note that every vertex $v \in V(C')\setminus V(S)$ of girth at most four has $|L'(v)| \geq 3$: to see this, note that if $v$ has girth three in $G$, then $|L(v)| \geq 5$; similarly, if $v$ has girth four in $G$, then $|L(v)| \geq 4$ and $v$ is adjacent to at most one of $v_5$ and $v_6$ in $G$.

Let $A' \subseteq V(C') \setminus V(S)$ be the set of vertices $v$ with $|L'(v)| \leq 2$. Note that $A'$ contains only vertices of girth at least five by the foregoing paragraph. It follows that $A'$ is an independent set, since every vertex in $A'$ is adjacent to one of $v_5$ and $v_6$ in $G$. Furthermore, since every vertex in $A'$ has girth at least five, $|L'(v)| = 2$ for all $v \in A'$. Thus $(G', L', S, A')$ is a canvas. Finally, no vertex in $A'$ is adjacent to either of $v_1$ or $v_4$, again because every vertex in $A'$ has girth at least five and is adjacent to one of $v_5$ and $v_6$. Thus $(G', L', S, A')$ is not exceptional of type (i) or (ii), and since $|V(S)| = 4$, $(G', L', S, A')$ is not exceptional of type (iii).  By the minimality of $K$, it follows that $G'$ admits an $L'$-colouring $\phi'$. By construction, $\phi' \cup \phi$ is an $L$-colouring of $G$, a contradiction. 
\end{proof}

If $v_{k+1} \in A$, then Lemma \ref{extraverts1} can be strengthened. This is shown below. 

\begin{lemma}\label{extraverts2}
If $v_{k+1} \in A$, then $q \neq k+3$.
\end{lemma}
\begin{proof}
Suppose not. Note that we may assume $|V(S)| = 4$, as otherwise since $\g(v_{k+1}) \geq 5$, we obtain a contradiction to one of Lemmas \ref{sep4cycle}-\ref{sep6cycle}. We break into cases, depending on whether or not $v_{k+3}$ is in $A$.

\vskip 4mm
\noindent
\textbf{Case 1: $v_{k+3}$ is in $A$.} In this case, let $\phi$ be a colouring of $C$. Note that $\phi$ exists, since $C$ is chordless by Lemma \ref{chordless} and $v_{k+2} \not \in A$. Thus $V(\Int(C)) \neq \emptyset$. Let $G'$ be the graph obtained from $G$ by deleting $v_{k+1}$, $v_{k+2}$, and $v_{k+3}$. Let $L'$ be the list assignment for $G'$ obtained from $L$ by setting $L'(v) = \{\phi(v)\}$ for every $v \in V(S)$ and $L'(v) = L(v) \setminus \{\phi(x): x \in N_G(v) \cap \{v_{k+1}, v_{k+2}, v_{k+3}\}\}$ for every $v \in V(G') \setminus V(S)$.
Let $C'$ be the graph whose vertex- and edge-set are precisely those of the outer face boundary of $G'$, and let $A'$ be the set of vertices in $V(C')\setminus V(S)$ with lists of size at most two under $L'$. We now show that $K' =(G', L', S, A')$ is a canvas. Note that every vertex $x \in V(\Int(C))$ is adjacent to at most one of $v_{k+1}$, $v_{k+2}$, and $v_{k+3}$: this follows immediately from the fact that both $v_{k+1}$ and $v_{k+3}$ are in $A$, and so by definition have girth at least five. Thus $|L'(x)| \geq 4$ for every vertex $x$ of girth three in $V(C') \setminus V(S)$. Similarly, $|L'(x)| \geq 3$ for every vertex $x$ of girth four in $V(C') \setminus V(S)$.  Thus $A'$ contains only vertices of girth at least five, and moreover every vertex $v$ in $A'$ satisfies $|L'(v)| = 2$. It remains only to show that $A'$ is an independent set. Suppose not: then there exist vertices $a_1, a_2$ in $A'$ such that $\{a_1a_2, a_1v_{k+3}, a_2v_{k+1}\} \subset E(G)$. But by Lemma \ref{sep5cycle}, $V(\Int(a_1a_2v_{k+1}v_{k+2}v_{k+3}a_1)) = \emptyset$ since $\g(a_1) \geq 5$. This implies that $\deg(v_{k+2}) = 2$, which is of course a contradiction since $|L(v_{k+2})| = 3$ and $K$ is a vertex-minimum counterexample.
Thus $K'$ is a canvas. We now show it is unexceptional. It is not exceptional of type (iii) since $|V(S)| = 4$.  Moreover, it is not exceptional of type (i) since every vertex in $A'$ has girth at least five and is adjacent to one of $v_{k+1}, v_{k+2},$ and $v_{k+3}$ in $G$. Finally, as noted above each vertex $x \in V(C')\setminus V(S)$ with $\g(x) = 3$ has $|L'(x)| \geq 4$. Since $K$ is unexceptional, it thus follows that $K'$ is not an exceptional canvas of type (ii).
\vskip 4mm
\noindent
\textbf{Case 2:  $v_{k+3}$ is not in $A$.} Since $A$ is an independent set and $v_{k+1} \in A$, it follows that $v_{k+2} \not \in A$ and so that $|L(v_{k+2})| = 3$. Thus $v_{k+2}$ contains a colour $c$ not in the list of $v_{k+1}$. Let $\phi$ be a colouring of $v_{k+2}$ and $v_{k+3}$, where $\phi(v_{k+2}) = c$ and $\phi(v_{k+3}) \in L(v_{k+3}) \setminus (\{c\} \cup L(v_{1}))$.  Let $G'$ be obtained from $G$ by deleting $v_{k+2}$ and $v_{k+3}$. Let $L'$ be the list assignment for $G'$ obtained from $L$ by setting $L'(v) = \{\phi(v)\}$ for every $v \in V(S)$ and $L'(v) = L(v) \setminus \{\phi(x): x \in N_G(v) \cap \{v_{k+2}, v_{k+3}\}\}$ for every $v \in V(G') \setminus V(S)$. Note that by our choice of $\phi$, we have that $L(v_{k+1}) = L'(v_{k+1})$. Let $A'$ be the set of vertices $v \in V(G')\setminus V(S)$ with $|L'(v)| \leq 2$, and let $C'$ be the graph whose vertex- and edge-set are precisely those of the outer face boundary of $G'$. As in the previous case, we will argue that $(G', L', S, A')$ is an unexceptional canvas. First, note that since every vertex $v \in V(C') \setminus V(S)$ of girth three has $|L(v)| \geq 5$, it follows that for each such vertex $|L'(v)| \geq 3$. Similarly, since every vertex $v \in V(C') \setminus V(S)$ with $\g(v) = 4$ is adjacent to at most one of $v_{k+2}$ and $v_{k+3}$ in $G$, it follows that each such vertex has $|L'(v)| \geq 4$. Thus every vertex in $A'$ has girth at least five. Note furthermore that every vertex in $A'$ has at least two colours in its list under $L'$, since $c \not \in L(v_{k+1})$ and every vertex in $A'$ is adjacent to at most one of $v_{k+2}$ and $v_{k+3}$. We claim moreover that $A'$ is an independent set: again, this follows easily from the fact that every vertex in $A'$ has girth five and is adjacent to one of $v_{k+2}$ and $v_{k+3}$. 

Thus $(G', L', S, A')$ is a canvas; it remains only to show it is unexceptional. That it is not exceptional of type (i) follow from the facts that $v_{k+1} \in A'$ and $C$ is chordless (thus $v_1v_{k+1} \not \in E(G)$), and that no vertex in $A' \setminus A$ is adjacent to both $v_1$ and $v_4$ since every vertex in $A' \setminus A$ is adjacent to one of $v_{k+2}$ and $v_{k+3}$. Suppose now that $(G', L', S, A')$ is exceptional of type (ii). In this case, there exists a vertex $v \in V(\Int(C))$ of girth three that is adjacent to a vertex $u$ in $A'$ and $|L'(v)| = 3$. Since $|L(v)| \geq 5$, it follows that $v$ is adjacent to both $v_{k+3}$ and $v_{k+2}$. Since every vertex in $A'$ is adjacent to one of $v_{k+3}$ and $v_{k+2}$ in $G$ and has girth at least five, this is a contradiction. Finally, $(G', L', S, A')$ is not an exceptional canvas of type (iii) since $S$ contains four vertices. By the minimality of $K$ and the fact that $|V(G')| < |V(G)|$, we have that $G'$ admits an $L'$-colouring $\phi'$. This again is a contradiction, since $\phi \cup \phi'$ is an $L$-colouring of $G$.  
\end{proof}

 We make the following definitions for convenience.

\begin{definition}
Let $k'$ be an index defined as follows: if $v_{k+1} \in A$, then $k' = k+1$. Otherwise, $k' = k$.
\end{definition} 
\begin{definition}
The \emph{available colour} $c$ at $v_{k'}$ is defined as follows: if $k = k'$, then $c \in L(v_{k'})$, and if $k' = k+1$, then $c \in  L(v_{k'}) \setminus L(v_k)$.
\end{definition}
By Observation \ref{listsizes}, $|L(v_k)| = 1$ and if $v_{k+1} \in A$, then $|L(v_{k'})| = 2$. We will show below in Lemma \ref{colours} (1) that the available colour at $v_{k'}$ is uniquely determined.

By Lemmas \ref{extraverts1} and \ref{extraverts2}, there exist vertices $v_{k'+1}, v_{k'+2}, v_{k'+3} \in V(C) \setminus V(S)$. The following  lemma partially describes their list assignments.

\begin{lemma}\label{colours} The following hold.
\begin{enumerate}[(1)]
\item $L(v_k) \subset L(v_{k+1})$,
\item $L(v_{k'+1}) \setminus \{c\} \subseteq L(v_{k'+2})$, where $c$ is the available colour at $v_{k'}$, and
\item $L(v_{k'+2}) \subseteq L(v_{k'+3})$.
\end{enumerate}
\end{lemma}
\begin{proof}[Proof of (1)]

Suppose not. That is, suppose $L(v_k) \not \subset L(v_{k+1})$. Recall that by Observation \ref{listsizes}, $|L(v_k)| = 1$. Let $G' = G-v_k$, and let $L'$ be a list assignment for $G'$ defined by $L'(v) = L(v) \setminus L(v_{k})$ for all $v \in N_G(v_k)$, and $L'(v) = L(v)$ for all $v \in V(G') \setminus N_G(v_k)$.  Note that $L(v_{k+1}) = L'(v_{k+1})$ since $L(v_k) \not \subset L(v_{k+1})$. Let $A'$ be the set of vertices in $V(G') \setminus V(S)$ that have lists of size at most two under $L'$. Since $L(v_{k+1}) = L'(v_{k+1})$, it follows that $v_{k+1} \not \in A'\setminus A$. Since $C$ is chordless by Lemma \ref{chordless}, it follows that  $A' \setminus A \subseteq V(\Int(C))$.  Thus every vertex $v \in A' \setminus A$ is a neighbour of $v_k$ in $G$ and has $|L(v)| = 3$, and so every vertex $v \in A' \setminus A$ satisfies both $|L'(v)| = 2$ and $\g(v) \geq 5$. Since every vertex in $A$ has girth at least five by definition, it follows that every vertex in $A'$ has girth at least five. Since every vertex in $A' \setminus A$ is adjacent to $v_k$ in $G$ and is in $V(\Int(C))$, it follows from Lemma \ref{intg4s} that $A'$ is an independent set. Thus $(G', L', S-v_k, A')$ is a canvas. Note that $S-v_k$ is an acceptable path in $G'$, since a subpath of an acceptable path is itself acceptable. Furthermore, note that $|V(S-v_k)| \leq 3$, and so $(G', L', S-v_k, A')$ is not an exceptional canvas of type (i) or (ii). If $|V(S-v_k)| = 3$, then by the definition of acceptable path since $S$ is acceptable $S-v_k$ contains a vertex of girth at least four. It follows that $(G', L', S-v_k, A')$ is not an exceptional canvas of type (iii), and is thus unexceptional. By the minimality of $K$, we have that $G'$ admits an $L'$-colouring $\phi$. But $\phi$ extends to an $L$-colouring of $G$ by setting $\phi(v_k) \in L(v_k)$, a contradiction.
\end{proof}
\begin{proof}[Proof of (2)]
Suppose not. That is, suppose there exists a colour $c' \in L(v_{k'+1}) \setminus (\{c\} \cup L(v_{k'+2})).$  Let $G' = G-v_{k'+1}$, and let $L'$ be a list assignment for $G'$ defined by $L'(v) = L(v) \setminus \{c'\}$ for all $v \in N_G(v_{k'+1})\setminus \{v_{k'}\}$, and $L'(v) = L(v)$ for all other $v \in V(G')$.  Let $A'$ be the set of vertices in $V(G') \setminus V(S)$ that have lists of size at most two under $L'$. Since $c' \not \in L(v_{k'+2})$, it follows that $L'(v_{k'+2}) = L(v_{k'+2})$. Since $C$ is chordless by Lemma \ref{chordless}, we have further that $A' \setminus A \subseteq V(\Int(C))$. We claim $K' = (G', L', S, A')$ is a canvas; the argument is identical to that in the proof of (1).  

First suppose $K'$ is unexceptional. Then by the minimality of $G$ we have that $G'$ admits an $L'$-colouring $\phi$. By the definition of available colour at $v_{k'}$, it follows that $\phi(v_{k'}) = c$. Since $c \neq c'$, we have that $\phi$ extends to an $L$-colouring of $G$ by setting $\phi(v_{k'+1}) = c'$, a contradiction. Thus we may assume that $K'$ is exceptional.

Suppose next that $K'$ is an exceptional canvas of type (i). Then there exists a vertex $u \in A'$ adjacent to $v_1$ and $v_4$. Since $K$ is unexceptional, we have that $u \not \in A$. Thus $u \in A' \setminus A$; but $v_1uv_4$ contradicts Lemma \ref{intg4s}, since $u \in V(\Int(C))$ and $\g(u) \geq 5$.

Next, suppose $(G', L', S, A')$ is an exceptional of type (ii). Then there exists a vertex $v \in V(G') \setminus V(S)$ with $\g(v) = 3$ and $|L'(v)|=3$ such that $v$ is adjacent to one of $v_2$ and $v_3$. Since $C$ is chordless, $v \in V(\Int(C))$, and so $|L(v)| \geq 5$. But then $|L'(v)| \geq 4$, a contradiction. 

Thus we may assume that $K'$ is an exceptional canvas of type (iii), and therefore that $G'$ contains a generalized wheel $W$ such that the vertices on the outer cycle of $W$ are on the outer cycle of $G'$ and have lists of size at most three under $L'$. Since $K$ is unexceptional, it follows that there exists a vertex $w$ on the outer cycle of $W$ that is not on the outer cycle of $G$. Since $w \in V(W)$, we have that $\g(w) = 3$ and so $|L(w)| \geq 5$. But then $|L'(w)| \geq 4$, a contradiction.
\end{proof}
\begin{proof}[Proof of (3)]
Suppose not. That is, suppose there exists a colour $c' \in L(v_{k'+2}) \setminus L(v_{k'+3})$. Let $\phi$ be an $L$-colouring of $v_{k'+2}$ and $v_{k'+1}$ where $\phi(v_{k'+2}) = c'$ and where $\phi(v_{k'+1}) \in L(v_{k'+1}) \setminus \{c, c'\}$, where $c$ is the available colour at $v_{k'}$. Note that since $A$ is an independent set, $|L(v_{k'+1})| = 3$ and so $\phi$ exists as described.

Let $G'$ be the graph obtained from $G$ by deleting $v_{k'+1}$ and $v_{k'+2}$, and let $C'$ be the subgraph of $G'$ with vertex-set and edge-set equal to that of the outer face boundary walk of $G'$. Let $L'$ be a list assignment obtained from $L$ by setting $L'(v_{k'}) = L(v_{k'})$, and $L'(v) = L(v) \setminus \{\phi(v_i) : i \in \{k'+1, k'+2\} \textnormal{ and } v_i \in N(v)\}$ for all $v \in V(G')\setminus \{v_{k'}\}$. Let $A'$ be the set of vertices in $V(G') \setminus V(S)$ with lists of size at most two under $L'$. Note that by assumption, $L'(v_{k'+3}) = L(v_{k'+3})$ and so since $C$ is chordless by Lemma \ref{chordless} it follows that every vertex $v \in V(C) \setminus \{v_{k'+1}, v_{k'+2}\}$ satisfies $L(v) = L'(v)$. We claim that $K' = (G', L', S, A')$ is a canvas. To see this, note that every vertex $v \in V(\Int(C))$ with $\g(v) = 3$ satisfies $|L(v)| \geq 5$ and thus $|L'(v)| \geq 3$. Similarly, every vertex $v \in V(\Int(C))$ of girth four is adjacent to at most one of $v_{k'+1}$ and $v_{k'+2}$ in $G$ and so satisfies $|L'(v)| \geq 3$. It follows that every vertex  $v \in A'$ has girth at least five and so has exactly two colours in its list under $L'$. Finally, we note that $A'$ is an independent set. To see this, observe that every vertex in $A'\setminus A$ has girth at least five and is adjacent to one of $v_{k'+1}$ and $v_{k'+2}$ in $G$. Thus $A'\setminus A$ is an independent set, and it follows from Lemma \ref{intg4s} that there are no edges between vertices in $A'\setminus A$ and vertices in $A$. This proves the claim that $K'$ is a canvas.

First suppose that $K'$ is unexceptional. By the minimality of $K$, we have that $G'$ admits an $L'$-colouring $\phi'$. But then $\phi \cup \phi'$ is an $L$-colouring of $G$, a contradiction. Thus we may assume $K'$ is exceptional.

Suppose next that $K'$ is an exceptional canvas of type (i). Then there exists a vertex $u \in A'$ adjacent to $v_1$ and $v_4$. Since $K$ is unexceptional, we have that $u \not \in A$. Thus $u \in A' \setminus A$; but $v_1uv_4$ contradicts Lemma \ref{intg4s} since $u \in V(\Int(C))$ and $\g(u) \geq 5$.

Next, suppose that $K'$ is an exceptional canvas of type (ii). Then there exists a vertex $w \in V(G') \setminus V(S)$ with $\g(w) = 3$ and $|L'(w)|=3$ such that $w$ is adjacent to one of $v_2$ and $v_3$ and to a vertex $u \in A'$. Since $C$ is chordless, $v \in V(\Int(C))$. Since $|L'(w)| = 3$ and $w \in V(\Int(C))$, it follows that $w$ is a neighbour of both $v_{k'+1}$ and $v_{k'+2}$ in $G$. If $u \in A$, then $uwv_{k'+1}$ contradicts Lemma \ref{intg4s} since $\g(u) \geq 5$. Thus $u \in A' \setminus A$, and so it follows that $u$ is adjacent to a vertex $v$ in $\{v_{k'+1}, v_{k'+2}\}$. This is a contradiction, since $\g(u) \geq 5$ and $uvwu$ is a cycle of length 3. 

Thus we may assume $K'$ is an exceptional canvas of type (iii). Then $k=3$, and $S$ contains only vertices of girth three. Moreover, $G'$ contains a generalized wheel $W$ such that the vertices on the outer cycle of $W$ are on the outer cycle of $G'$ and have lists of size at most three under $L'$. Let $v_1v_2v_3w_1 \dots w_tv_1$ be the outer face boundary walk of $W$. Since $|L'(w_1)| = 3$, it follows that $|L(w_1)| \geq 3$. Thus $w_1 \not \in A$. Moreover, $w_1 \neq v_{k'+1}$ since $v_{k'+1} \not \in V(G')$. It follows that $w_1 \neq v_4$, and so that $w_1 \in V(\Int(C))$. Thus $|L(w_1)| \geq 5$; and since $|L'(w_1)| =3$, we have that $w_1$ is adjacent to both $v_{k'+1}$ and $v_{k'+2}$ in $G$. Since $5 \in \{k'+1, k'+2\}$, we have that $w_1v_5 \in E(G)$. Since $w_1$ is adjacent to $v_3$, it follows from Lemma \ref{nogenwheels} that $w_1$ is also adjacent to $v_4$ in $G$. This contradicts Lemma \ref{deg3forv4}. 
\end{proof}

We require one final corollary, which will be used to show that $G$ contains our main reducible configuration. The corollary shows that $A \cap \{v_{k+1}, v_{k+2}, v_{k+3}, v_{k+4}\} \not \in \{\{v_{k+1}, v_{k+4}\}, \{v_{k+3}\}\}.$ 

\begin{cor}\label{notthird}
 $v_{k'+3} \not \in A$.
\end{cor}
\begin{proof}
Suppose not. Since $v_{k'+3} \in A$, it follows that $|L(v_{k'+3})| = 2$. Since $A$ is an independent set, $v_{k'+2} \not \in A$ and so by Observation \ref{listsizes} we have that $|L(v_{k'+2})| = 3$. But this contradicts Lemma \ref{colours} (3).

\end{proof}

Since $A$ is an independent set, we thus have that if $v_{k+1} \in A$, then $A \cap \{v_{k+2}, v_{k+3}, v_{k+4}\} \in \{\emptyset, \{v_{k+3}\}\}$. Similarly, if $v_{k+1} \not \in A$, then $A \cap \{v_{k+1}, v_{k+2}, v_{k+3}\} \in \{\emptyset, \{v_{k+2}\}\}$. To summarize: $A \cap \{v_{k'+1}, v_{k'+2}, v_{k'+3}\} \in \{\emptyset, \{v_{k'+2}\}\}$.

We now define our main reducible configuration.  It follows from Lemmas \ref{extraverts1} and \ref{extraverts2} that $q \geq k'+3$. Recall that by Observation \ref{listsizes}, every vertex in $V(C) \setminus (V(S) \cup A)$ has a list of size three.

\begin{definition}\label{defdelpath}
 Let $P = v_{k'+1}v_{k'+2} \dots v_{j}$ be a subpath of $C$ satisfying the following set of conditions: 
\begin{enumerate}
    \item $j \in \{k'+3, \dots, q\}$,
    \item either $V(P) \cap A = \{v_{k'+2}\}$ or $V(P) \cap A = \emptyset$, 
    \item $L(v_{i-1}) \subseteq L(v_i)$ for all $k'+3 \leq i \leq j$, and
    \item $L(v_{j}) \not \subseteq L(v_{(j \mod q)+1})$.
\end{enumerate}
We call $P$ a deletable path. 
\end{definition}

By Corollary \ref{notthird}, we have that $A \cap \{v_{k'+1}, v_{k'+2}, v_{k'+3}\} \in \{\emptyset, \{v_{k'+2}\}\}$. Thus we immediately have the following.

\begin{cor}\label{delpath}
$G$ contains a deletable path.
\end{cor}

\section{A Proof of Our Main Theorem}
Before proceeding, we give a brief overview of this section. This section contains a proof of our main technical theorem, Theorem \ref{345colouring}. Recall that $K = (G, L, S, A)$ is a counterexample to Theorem \ref{345colouring}, where $|V(G)|$ is minimized over all counterexamples to the theorem, and subject to that, where $\sum_{v \in V(G)}|L(v)|$ is minimized. By Corollary \ref{delpath}, $G$ contains a deletable path $P = v_{k'+1}v_{k'+2} \dots v_{j}$. 

Most of the lemmas in this section take the following basic shape: we colour and delete $P$, and argue about the structure of what remains. In Lemma \ref{notnone}, we show that $V(P) \cap A = \{v_{k'+2}\}$ and argue that there exists a separating path as described in Lemma \ref{3-5-5-3}, where the endpoints of the path are $v_{k'+1}$ and $v_{k'+3}$. 

From there, Lemma \ref{P-structure-i} establishes that no vertex in $\Int(C)$ is adjacent to two vertices on $P$ at distance at least two in $P$. This will be useful in arguing that upon deleting $P$ and modifying lists where appropriate, what remains is a canvas: that is, no list loses too many colours. Finally, in Lemma \ref{notctx}, we colour and delete $P$ as well as the separating path described in Lemma \ref{notnone} and argue via a series of claims that what remains is an unexceptional canvas. By induction, this canvas admits an $L$-colouring \textemdash and so $G$ admits an $L$-colouring. This shows that $K$ is not a counterexample to Theorem \ref{345colouring}, completing the proof.

We begin with the following lemma.

\begin{lemma}\label{notnone}
$V(P) \cap A = \{v_{k'+2}\}$, and there exist distinct vertices $u_1, u_2 \in V(\Int(C))$ with $\g(u_1) \geq 5$ and $\g(u_2) \geq 5$ such that $u_1$ is adjacent to $v_{k'+1}$, $u_2$ is adjacent to $v_{k'+3}$, and $u_1$ is adjacent to $u_2$.
\end{lemma}
\begin{proof}
Suppose not. Let $\phi$ be an $L$-colouring of $P$ such that: $\phi(v_j) \not \in L(v_{(j \mod q)+1})$;  $\phi(v_{k'+1}) \in L(v_{k'+1}) \setminus \{c\}$, where $c$ is the available colour at $v_{k'}$; if $v_{k'+2} \not \in A$, then $\phi$ is a $2$-colouring of $P -v_{k'+1}$; and if $v_{k'+2} \in A$, then $\phi$ is a 2-colouring of $P-v_{k'+1}-v_{k'+2}$. Note that $\phi$ exists: by Observation \ref{listsizes} and the definition of deletable path we have that $L(v_{k'+3}) = L(v_{i})$ for all $i \in \{k'+4, \dots, j\}$, and if $v_{k'+2} \not \in A$, then $L(v_{k'+2}) = L(v_{k'+3})$ as well. By Observation \ref{listsizes} and Lemma \ref{colours} (1), the available colour at $v_{k'}$ is unique.

Let $G'$ be the graph obtained from $G$ by deleting $V(P)$, and let $L'$ be the list assignment for $G'$ obtained from $L$ by setting $L'(v_{k'}) = L(v_{k'})$ and $L'(v) = L(v) \setminus \{\phi(x) : x \in V(P)\cap N_G(v)\}$ for all $v \in V(G') \setminus \{v_{k'}\}$. Let $C'$ be the subgraph of $G'$ whose vertex- and edge-set are precisely those of the outer face boundary walk of $G'$, and let $A'$ be the set of vertices in $V(G') \setminus V(S)$ with lists of size at most two under $L'$. 

We now show the following.

\begin{claim}\label{canvas}
$(G', L', S, A')$ is a canvas.
\end{claim}
\begin{proof}
We proceed via a series of subclaims. Note that by our choice of $L'$, every vertex $v \in V(S)$ satisfies $|L'(v)| = 1$.  \vskip 3mm
\noindent
\textbf{Subclaim a)} \emph{Every vertex $v \in V(C') \setminus V(S)$ with $\g(v) = 3$ satisfies $|L'(v)| \geq 3$}.
\begin{proof} Suppose not. Note that if $v \in V(C) \setminus V(S)$, then since $C$ is chordless by Lemma \ref{chordless} and $\phi(v_j) \not \in L(v_{(j \mod q)+1})$ by our choice of $\phi$, it follows that $|L'(v)| \geq 3$, a contradiction. Thus we may assume $v \in V(C') \setminus V(C)$, and so that $v \in V(\Int(C))$. Since $\g(v) =3$, we have by the definition of canvas that $|L(v)| \geq 5$. Since $|L'(v)| < 3$, it follows that $v$ is adjacent to at least three vertices in $V(P)$; and in particular, since $\phi$ is a 2-colouring of $P - v_{k'+2}-v_{k'+1}$, that either $vv_{k'+1} \in E(G)$ or $vv_{k'+2} \in E(G)$. Let $\ell$ be the smallest index such that $vv_\ell \in E(G)$ and $\ell \in \{k'+1, k'+2\}$. Let $m$ be the largest index such that $v_m \in N_G(v) \cap V(P)$. Since $v$ neighbours at least three vertices in $P$, it follows that $m \geq \ell+2$. By Lemma \ref{nogenwheels}, since $v$ is adjacent to $v_{\ell}$ and $v_m$, we have that $v$ is also adjacent to $v_{\ell+1}, v_{\ell+2}, \dots, v_{m-1}$. Thus for each $i \in \{\ell, \dots, m\}$, we have that $\g(v_i) = 3$. Since $i \leq k'+2 < m$, it follows that $\g(v_{k'+2}) = 3$ and so that $v_{k'+2} \not \in A$. Thus by the definition of deletable path, $L(v_{k'+2}) = L(v_{k'+3})$. It then follows from Lemma \ref{nobigfans} that $\ell = k'+1$ and $m = k'+3$. 

Hence $v$ is adjacent to $v_{k'+1}$, $v_{k'+2}$, and $v_{k'+3}$. Note that by Lemma \ref{sep3cycle}, $V(\Int(vv_{k'+1}v_{k'+2}v)) = V(\Int(vv_{k'+2}v_{k'+3}v)) = \emptyset$. Let $G''$ be the graph obtained from $G$ by identifying $v_{k'+1}$ and $v_{k'+3}$ to a new vertex $z$ and deleting $v_{k'+2}$. Let $L(z) = L(v_{k'+1})$. Note that by Lemma \ref{colours} (2) and (3), we have that $L(v_{k'+1}) \setminus \{c\} \subseteq L(v_{k'+3})$, where $c$ is the available colour at $v_{k'}$. By Lemma \ref{closeafan}, for every vertex $x \in V(G'')$, if $\g_G(x) \geq 5$, then $g_{G''}(x) \geq 5$, and similarly if $\g_G(x) = 4$, then $g_{G''}(x) \geq 4$. It follows that $S$ is an acceptable path in $G''$. Note that $K'' = (G'', L, S, A)$ is a canvas; in particular, $|L(z)| = 3$. 

First suppose $K''$ is unexceptional. By the minimality of $K$, we have that $K''$ admits an $L$-colouring $\varphi$. Note that by definition of available colour, $\varphi(z) \neq c$, where $c$ is the available colour at $v_{k'}$. But then $\varphi(z) \in L(v_{k'+3})$, and so $\varphi$ extends to an $L$-colouring of $G$ by setting $\varphi(v_{k'+1}) = \varphi(v_{k'+3}) = \varphi(z)$ and $\varphi(v_{k'+2}) \in L(v_{k'+2}) \setminus \{\varphi(z), \varphi(v)\}$, a contradiction.

Thus we may assume that $K''$ is an exceptional canvas. Since $C$ is chordless by Lemma \ref{chordless}, $z \not \in A$, and $K$ is unexceptional, we have that $K''$ is not an exceptional canvas of type (i) or (ii). We may therefore assume that $K''$ is an exceptional canvas of type (iii), and thus that $G''$ contains a subgraph $W$ that is a generalized wheel with principal path $S$ such that the vertices on the outer cycle of $W$ are on the outer face boundary of $G''$ and all have lists of size at most three under $L$. Again because $C$ is chordless it follows that the outer cycle of $W$ is the outer cycle of $G''$, and that the outer cycle of $G''$ is also chordless. Since every generalized wheel that is neither a triangle nor a wheel has a chord in its outer cycle, it follows that $W$ is either a triangle or a wheel. Note that since $K$ is unexceptional, we have that $z \in V(W)$. Since $|V(W)| \geq |V(S)| + |\{z\}| = 4$, we have that $W$ is a wheel. It follows that every triangle in $G''$ corresponds to a triangle in $G$ (replacing $z$ by $v_{k'+1}$ or $v_{k'+3}$ where appropriate). By Lemma \ref{sep3cycle}, we have that $W = G''$; and since $v \in V(\Int(C))$, it follows further that $v$ is the only vertex in $W$ not in the outer cycle of $G''$. But then $G$ too is a wheel with principal path $S$, and since $|L(v_{k'+2})| = 3$ by Observation \ref{listsizes}, it follows that $K$ is an exceptional canvas of type (iii), a contradiction. Thus we may assume that every vertex in $V(C') \setminus V(S)$ of girth three has a list of size at least three under $L'$. 
\end{proof}

\noindent
\textbf{Subclaim b)} \emph{Every vertex $v \in V(C') \setminus V(S)$ with $\g(v) = 4$ satisfies $|L'(v)| \geq 3$}.
\begin{proof}
Suppose not. Let $v \in V(C')\setminus V(S)$ satisfy $|L'(v)| \leq 2$ and $\g(v) = 4$. Note that by our choice of $\phi$ and the fact that $C$ is chordless by Lemma \ref{chordless}, we have that $L'(v) = L(v)$ for all $v \in V(C)$. Thus we may assume that $v \in V(C') \setminus V(C)$, and so that $v \in V(\Int(C))$. Thus $|L(v)| \geq 4$. By Lemma \ref{intg4s}, since $\g(v) =4$ we have that $v$ is adjacent to at most one vertex in $V(P)$  and thus since $|L(v)| \geq 4$, it follows that $|L'(v)| \geq 3$, a contradiction. Thus we may assume that every vertex in $V(C') \setminus V(S)$ of girth four has a list of size at least three under $L'$.
\end{proof}

\noindent
\textbf{Subclaim c)} \emph{Every vertex $v \in A'$ satisfies $\g(v) \geq 5$ and $|L'(v)| = 2$.}
\begin{proof}
It follows from Subclaims a) and b) that every vertex in $A'$ has girth at least five. Suppose now that there exits a vertex $v \in A'$ with $|L'(v)| \leq 1$. Note that if $v \in V(C)$, then by our choice of $\phi$ and the fact that $C$ is chordless by Lemma \ref{chordless} it follows that $L(v) = L'(v)$. Thus we may assume that $v \in V(C')\setminus V(C)$, and so that $v \in V(\Int(C))$. By Lemma \ref{intg4s}, since $\g(v) \geq 5$ we have that $v$ is adjacent to only one vertex in $V(P)$ and so that $|L'(v)| \geq 2$, a contradiction. Thus we may assume that every vertex in $V(C') \setminus V(S)$ of girth at least five has a list of size at least two under $L'$.
\end{proof}

Given Subclaims a), b), and c), to prove that $K'$ is a canvas it remains only to show that that $A'$ is an independent set. To see this, suppose not. Then there exist distinct vertices $u_1, u_2 \in A'$ such that $u_1u_2 \in E(G)$. (Note that $A' \cap V(P) = \emptyset$, and so that in particular $v_{k'+2} \not \in \{u_1, u_2\}$.) Since $C$ is chordless, $\{u_1, u_2\} \not \subseteq A$. Suppose that exactly one of $u_1$ and $u_2$ is in $A$. This contradicts Lemma \ref{intg4s}, since $\g(u_i) \geq 5$ for each $i \in \{1,2\}$ and every vertex in $A' \setminus A$ is adjacent to a vertex in $V(P)$. Thus we may assume that $\{u_1, u_2\} \subseteq A' \setminus A$, and so that $\{u_1, u_2\} \subseteq V(\Int(C))$. Recall that by the definition of deletable path, either $V(P) \cap A = \emptyset$ or $V(P) \cap A = \{v_{k'+2}\}$. Since $\g(u_1) \geq 5$ and $\g(u_2) \geq 5$ and each of $u_1$ and $u_2$ is adjacent to a vertex in $V(P)$, it follows from Lemma \ref{3-5-5-3} that $V(P) \cap A = \{v_{k'+2}\}$, and that, up to relabelling $u_1$ and $u_2$, both $u_1v_{k'+1} \in E(G)$ and $u_2v_{k'+3} \in E(G)$. Thus the hypotheses of Lemma \ref{notnone} hold, a contradiction.
\end{proof}

Thus $K' = (G', L', S, A')$ is a canvas by Claim \ref{canvas}. First suppose $K'$ is unexceptional. Then by the minimality of $K$ we have that $K'$ admits an $L'$-colouring $\phi'$. But then $\phi \cup \phi'$ is an $L$-colouring of $G$, a contradiction. We may therefore assume that $K'$ is an exceptional canvas.

Suppose next that $K'$ is an exceptional canvas of type (i). Then there exists a vertex $u \in A'$ adjacent to $v_1$ and $v_4$. Since $K$ is unexceptional, $u \not \in A$. Thus $u \in A'\setminus A$, and so $u \in V(\Int(C))$. But then $v_1uv_4$ contradicts Lemma \ref{intg4s}, since $\g(u) \geq 5$. 

Next suppose that $K'$ is exceptional of type (ii). Then $|V(S)| = 4$, and moreover $G'$ contains a vertex $w \in V(G')\setminus V(S)$ of girth three with $|L'(w)| = 3$ adjacent to one of $v_2$ and $v_3$, and a vertex $u \in A'$ adjacent to $w$ and one of $v_4$ and $v_1$. Since $|L'(w)| = 3$ and $C$ is chordless by Lemma \ref{chordless}, it follows that $w \in V(C')\setminus V(C)$, and so that $w$ is adjacent in $G$ to at least two vertices on $P$. If $u \in A$, then $w$ contradicts Lemma \ref{intg4s} since $w$ is adjacent to $u$ and a vertex on $P$ and $\g(u) \geq 5$. Thus we may assume $u \in A'\setminus A$. But then since $u$ is adjacent to one of $v_1$ and $v_4$, again this contradicts Lemma \ref{intg4s}, since every vertex in $A'\setminus A$ is adjacent to a vertex on $P$ and $\g(u) \geq 5$. 

Thus we may assume $K'$ is an exceptional canvas of type (iii). Then $k=3$, and $S$ contains only vertices of girth three. Moreover, $G'$ contains a generalized wheel $W$ such that the vertices on the outer cycle of $W$ are on the outer cycle of $G'$ and have lists of size at most three under $L'$. Let $v_1v_2v_3w_1 \dots w_tv_1$ be the outer cycle of $W$. Since $|L'(w_1)| = 3$, it follows that $|L(w_1)| \geq 3$. Thus $w_1 \not \in A$. Moreover, $w_1 \neq v_{k'+1}$ since $v_{k'+1} \not \in V(G')$. It follows that $w_1 \neq v_4$, and so since $C$ is chordless by Lemma \ref{chordless} we have that $w_1 \in V(\Int(C))$. Since $\g(w_1) = 3$, we have further that $|L(w_1)| \geq 5$; and since $|L'(w_1)| =3$, we have that $w_1$ is adjacent to at least two vertices in $V(P)$. Since $w_1$ is also adjacent to $v_3$, it follows from Lemma \ref{nogenwheels} that $w$ is adjacent to $v_{4}$ and $v_5$ in $G$. This contradicts Lemma \ref{deg3forv4}. 
\end{proof}

Following Lemma \ref{P-structure-i}, our final lemma \textemdash Lemma \ref{notctx}\textemdash will also involve colouring and deleting $P$, restricting lists where appropriate, and arguing that what remains is an unexceptional canvas. Lemma \ref{P-structure-i} shows that in doing so, the vertices in $\Int(C)$ lose at most one colour from their list.

\begin{lemma}\label{P-structure-i}
There does not exist a vertex $v \in V(\Int(C))$ such that $v$ is adjacent to two vertices in $V(P)$ at distance at least two in $P$.
\end{lemma}
\begin{proof}
Suppose not, and let $v$ be a vertex adjacent to two vertices of $P$ at distance at least two. Let $i$ be the smallest index with $k'+1 \leq i \leq j-2$ such that $v_i \in N_G(v) \cap V(P)$. Let $m$ be the largest index with $k'+3 \leq m \leq j$ such that $v_m \in N_G(v) \cap V(P)$.  By Lemma \ref{nogenwheels}, $v$ is adjacent to every vertex in $\{v_i, v_{i+1}, \dots, v_m\}$: that is, $vv_iv_{i+1}\dots v_mv$ is the outer cycle of a broken wheel with principal path $v_mvv_i$. By Lemma \ref{notnone}, $V(P) \cap A = \{v_{k'+2}\}$. Since $\g(v_{k'+2}) \geq 5$, it follows that $i \geq k'+3$. By the definition of $P$, we have that $L(v_m) = L(v_i)$. Since $m \geq i+2$, this contradicts Lemma \ref{nobigfans}.
\end{proof}

\begin{figure}[ht]
\begin{center}
\hskip 6mm
\begin{tikzpicture}[scale=0.80]
   \newdimen\R
   \R=6cm
   \draw (0:\R) \foreach \x in {0,25,50,75,100, 120} { -- (\x:\R) };
   \foreach \x/\l/\p in
     { 25/{$v_{k'+3}$}/right,
      75/{$v_{k'+1}$}/right
     }
     \node[inner sep=1pt, minimum size=7pt, circle,draw,fill=white,label={\p:\l}] (\x) at (\x:\R) {};
   \node[inner sep=1pt, minimum size=7pt, circle, draw, fill, label={above:$v_{k'}$}] (100) at (100:\R) {}; 
   \node[inner sep=1pt, minimum size=7pt, star,star points=4,star point ratio=0.5, draw, fill=white, label={right:$v_{k'+2}$}] (50) at (50:\R) {}; 
   \newdimen\r
   \r=3cm
   \foreach \x/\l/\p in
     { 15/{$u_{2}$}/right,
      85/{$u_{1}$}/right
     }
     \node[inner sep=1pt, minimum size=7pt, circle,draw,fill=white,label={\p:\l}] (\x) at (\x:\r) {};
     \draw (15)--(25);
     \draw (15)--(85);
     \draw (85)--(75);
\end{tikzpicture}
\hskip 6mm
\begin{tikzpicture}[scale=0.80]
   \newdimen\R
   \R=6cm
   \draw (0:\R) \foreach \x in {0,20,40,60,80,100,120} {  -- (\x:\R) };
   \foreach \x/\l/\p in
     { 20/{$v_{k'+3}$}/right,
      60/{$v_{k'+1}$}/above
     }
     \node[inner sep=1pt, minimum size=7pt, circle, draw, fill=white, label={\p:\l}] (\x) at (\x:\R) {};
   \foreach \x/\l/\p in
     { 40/{$v_{k'+2}$}/right,
      80/{$v_{k'}$}/above
     }
     \node[inner sep=1pt, minimum size=7pt, star,star points=4,star point ratio=0.5, draw, fill=white, label={\p:\l}] (\x) at (\x:\R) {};
   \node[inner sep=1pt, minimum size=7pt, circle, draw, fill, label={above:$v_{k}$}] (100) at (100:\R) {}; 
   \newdimen\r
   \r=3cm
   \foreach \x/\l/\p in
     { 10/{$u_{2}$}/right,
      70/{$u_{1}$}/right
     }
     \node[inner sep=1pt, minimum size=7pt, circle,draw,fill=white,label={\p:\l}] (\x) at (\x:\r) {};
     \draw (10)--(20);
     \draw (10)--(70);
     \draw (70)--(60);
\end{tikzpicture}
\caption{The final cases considered in Lemma \ref{notctx}. Vertices in $S$ are black; vertices in $A$ are drawn as four-pointed stars. On the left, $v_{k+1} \not \in A$ and so $k = k'$. On the right, since $v_{k+1} \in A$, it follows that $k' = k+1$. Recall that by definition, the deletable path $P$ begins at $v_{k'+1}$; recall moreover that by Lemma \ref{sep5cycle}, since $\g(u_1)=5$ it follows that the 5-cycles in the figure have no vertices in their interior.}
    \label{fig:final_arg}
\end{center}
\end{figure}
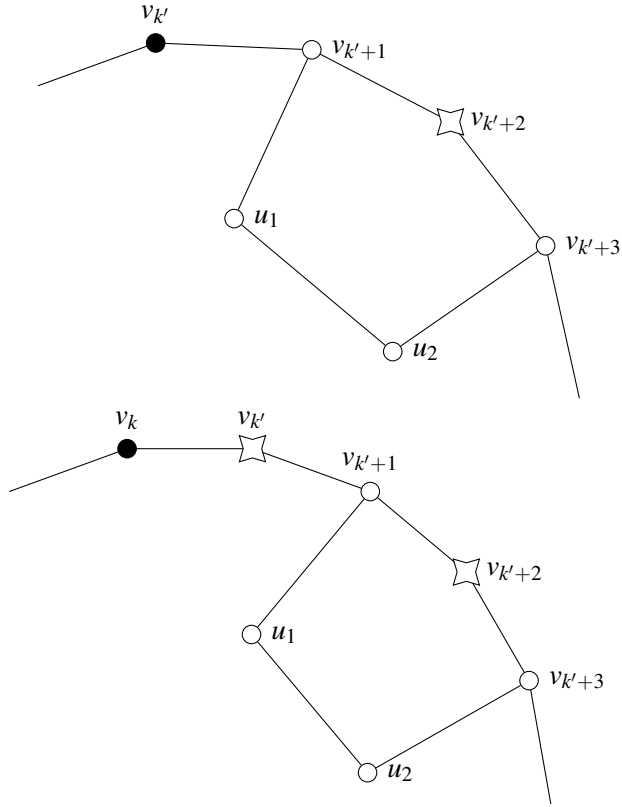

The following lemma concludes the proof of Theorem \ref{345colouring} (and moreover the paper). See Figure \ref{fig:final_arg} for an illustration of the cases considered in Lemma \ref{notctx}.
\begin{lemma}\label{notctx}
$K$ is not a counterexample to Theorem \ref{345colouring}.
\end{lemma}
\begin{proof}
Suppose not. By Corollary \ref{notthird}, $G$ contains a deletable path $P$, and by Lemma \ref{notnone}, $V(P) \cap A = \{v_{k'+2}\}$ and there exists an edge $u_1u_2$ in $E(\Int(C))$ such that both $u_1$ and $u_2$ have girth at least five, and $\{u_1v_{k'+1}, u_2v_{k'+3}\} \subset E(G)$. By Lemma \ref{sep5cycle}, we have that $\Int(u_1u_2v_{k'+3}v_{k'+2}v_{k'+1}u_1) = \emptyset$, and so that $N_G(v_{k'+2}) = \{v_{k'+1}, v_{k'+3}\}$.
We first show the following. 

\begin{claim}\label{colexists}
There exists an $L$-colouring $\phi$  of $G[\{u_1, u_2\} \cup V(P) \cup V(S)]$ such that:
\begin{itemize}
    \item $\phi(v_{j}) \not \in L(v_{(j \mod q)+1})$, and
    \item $\phi(v_{k'+1})$ is not the available colour at $v_{k'}$.
\end{itemize}
\end{claim}
\begin{proof}
Note that by Lemma \ref{intg4s}, since $\g(u_1) \geq 5$ it follows that $v_{k'+1}$ is the unique neighbour of $u_1$ in $V(P) \cup \{v_1, v_k\}$. Similarly, since $\g(u_2) \geq 5$, we have that $v_{k'+3}$ is the unique neighbour of $u_2$ in $V(P) \cup \{v_1, v_k\}$. Since each of $u_1$ and $u_2$ have girth at least five, it follows that at most one of $u_1$ and $u_2$ has a neighbour in $V(S)$.

Since $|L(v_{k'+1})| = 3$ and $C$ is chordless by Lemma \ref{chordless}, we have that $G[V(P) \cup V(S)]$ admits an $L$-colouring $\phi$ as described in the statement of the claim (by first colouring $S$ and then the vertices of $P$ in decreasing order of index). Moreover, as argued above there exists $i \in \{1,2\}$ such that $u_i$ has degree at most two in $H = G[\{u_1, u_2\} \cup V(P) \cup V(S)]$. After colouring $G[ V(P) \cup V(S)]$, we then colour $u_{3-i}$ with a colour $\phi(u_{3-i}) \in L(u_{3-i})\setminus \{ \phi(x): x \in N_H(u_{3-i})\}$. Note that $u_{3-i}$ has at most two neighbours in $H-u_i$, and so since $|L(u_{3-i})|\geq 3$, we have that $\phi(u_{3-i})$ exists. Finally, we colour $u_{i}$ with a colour in $L(u_i) \setminus\{ \phi(x): x \in N_H(u_{i})\}$. Since $u_i$ has at most two neighbours in $H$ and $|L(u_i)| \geq 3$, this is possible.
\end{proof}

Let $\phi$ be as in the statement of Claim \ref{colexists}, and let $G'$ be the graph obtained from $G$ by deleting $V(P) \cup \{u_1, u_2\}$. Note that by Lemma \ref{sep5cycle} and the fact that $\g(u_1) \geq 5$, we have that $\Int(v_{k'+1}v_{k'+2}v_{k'+3}u_2u_1v_{k'+1}) = \emptyset$. Let $C'$ be the graph whose vertex- and edge-set are precisely those of the outer face boundary of $G'$. Let $L'$ be the list assignment obtained from $L$ by setting $L'(v_{k'}) = L(v_{k'})$ and $L'(v) = L(v) \setminus \{\phi(x) : x \in (V(P)\cup \{u_1, u_2\}) \cap N(v)\}$ for all $v \in V(G') \setminus \{v_{k'}\}$. Let $A'$ be the set of vertices in $V(G') \setminus V(S)$ with lists of size at most two under $L'$.

Claims \ref{Csfine} and \ref{notadjto2} will be used repeatedly to argue that $(G', L', S, A')$ is a canvas.
\begin{claim}\label{Csfine}
Every vertex $v \in V(C')$ satisfies $L(v) = L'(v)$.
\end{claim}
\begin{proof}
Let $v \in V(C')$. If $v \in V(S)$, then $L(v) = L'(v)$ since $\phi$ is a colouring of $G[\{u_1, u_2\} \cup V(P) \cup V(S)]$. Thus we may assume $v \in V(C')\setminus V(S)$. Note that that since $u_1v_{k'+1} \in E(G)$ and $\g(u_1) \geq 5$, it follows from Lemma \ref{intg4s} that $u_1$ is not adjacent to $v$. Similarly, $u_2$ is not adjacent to $v$. By Lemma \ref{chordless}, $C$ is chordless and so no internal vertex of $P$ is adjacent to $v$. Moreover, note that $L'(v_{k'}) = L(v_{k'})$ by definition of $L'$. Thus we may assume that $v = v_{(j \mod q) + 1}$, as otherwise $L(v) = L'(v)$. By Claim \ref{colexists}, $\phi(v_j) \not \in L(v_{(j \mod q) + 1})$, and so $L(v_{(j \mod q) + 1}) = L'(v_{(j \mod q) + 1})$, as desired.
\end{proof}

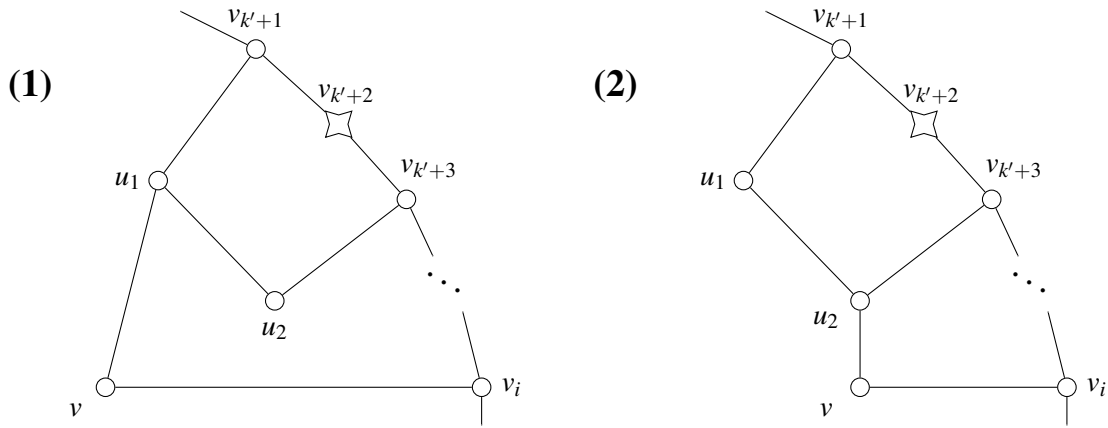
\begin{figure}[ht]
\tikzset{white/.style={shape=circle,draw=black,fill=white,inner sep=1pt, minimum size=7pt}}
\tikzset{starnode/.style={inner sep=1pt, minimum size=7pt, star,star points=4,star point ratio=0.5, draw, fill=white}}
\tikzset{invisible/.style={shape=circle,draw=black,fill=black,inner sep=0pt, minimum size=0.1pt}}

\begin{center}
\hskip 6mm
\begin{tikzpicture}
        \node[]  at (4,8.5){\Large{\textbf{(1)}}};
        \node[white] (1) at (7,9){};
        \node[] at (7, 9.4) {$v_{k'+1}$};
        \node[starnode] (2) at (8.1,8){};
        \node[]  at (8.2, 8.4){$v_{k'+2}$};
        \node[white] (3) at (9,7){};     
        \node[]  at (9.3, 7.4){$v_{k'+3}$};
        \node[invisible] (11) at (9.35, 6.25){};
        \node[]  at (9.5, 6){\Large{$\ddots$}};
        \node[invisible] (12) at (9.75, 5.5){};
        \node[white] (4) at (10,4.5){};
        \node[]  at (10.4, 4.5){$v_i$};
        \node[white] (6) at (7.25,5.65){}; 
        \node[]  at (7.25,5.25){$u_2$};
        \node[white] (7) at (5.7,7.25){}; 
        \node[]  at (5.3, 7.25){$u_1$};
        \node[white] (8) at (5,4.5){};  
        \node[]  at (4.6,4.2){$v$};
        \node[invisible] (9) at (6,9.5){};                
        \node[invisible] (10) at (10,4){};                
        
        \draw[black] (1)--(2); 
        \draw[black] (2)--(3);
        \draw[black] (3)--(11);
        \draw[black] (4)--(12);   
        \draw[black] (4)--(10); 
        \draw[black] (1)--(7);
        \draw[black] (1)--(9);
        \draw[black] (6)--(7);
        \draw[black] (6)--(3);  
        \draw[black] (8)--(7);   
        \draw[black] (8)--(4);

\end{tikzpicture}
\hskip 6mm
\begin{tikzpicture}
        \node[]  at (4,8.5){\Large{\textbf{(2)}}};
        \node[white] (1) at (7,9){};
        \node[] at (7, 9.4) {$v_{k'+1}$};
        \node[starnode] (2) at (8.1,8){};
        \node[]  at (8.2, 8.4){$v_{k'+2}$};
        \node[white] (3) at (9,7){};     
        \node[]  at (9.3, 7.4){$v_{k'+3}$};
        \node[invisible] (11) at (9.35, 6.25){};
        \node[]  at (9.5, 6){\Large{$\ddots$}};
        \node[invisible] (12) at (9.75, 5.5){};
        \node[white] (4) at (10,4.5){};
        \node[]  at (10.4, 4.5){$v_i$};
        \node[white] (6) at (7.25,5.65){}; 
        \node[]  at (6.8,5.4){$u_2$};
        \node[white] (7) at (5.7,7.25){}; 
        \node[]  at (5.3, 7.25){$u_1$};
        \node[white] (8) at (7.25,4.5){};  
        \node[]  at (6.8,4.2){$v$};
        \node[invisible] (9) at (6,9.5){};                
        \node[invisible] (10) at (10,4){};                
        
        \draw[black] (1)--(2); 
        \draw[black] (2)--(3);
        \draw[black] (3)--(11);
        \draw[black] (4)--(12);   
        \draw[black] (4)--(10); 
        \draw[black] (1)--(7);
        \draw[black] (1)--(9);
        \draw[black] (6)--(7);
        \draw[black] (6)--(3);  
        \draw[black] (8)--(6);   
        \draw[black] (8)--(4);

\end{tikzpicture}
\caption{Cases considered in Claim \ref{notadjto2}. Vertices in $A$ are drawn as four-pointed stars. Both $u_1$ and $u_2$ have girth five. Recall that by Lemma \ref{sep5cycle}, since $\g(u_1)=5$ it follows that the 5-cycles in all figures have no vertices in their interior.}
    \label{fig:45sfine}
\end{center}
\end{figure}

\begin{claim}\label{notadjto2}
There does not exist a vertex in $V(\Int(C))  \setminus \{u_1, u_2\}$ adjacent to a vertex in $V(P)$ and a vertex in  $\{u_1, u_2\}$.
\end{claim}
\begin{proof}
Suppose not, and let $v \in V(\Int(C)) \setminus \{u_1, u_2\}$ be a counterexample. Let $i$ be an index such that $k'+1 \leq i \leq j$ and $v$ is adjacent to $v_i$. Since $\g(u_2) \geq 5$ and $v$ is adjacent to one of $u_1$ and $u_2$, it follows that $i \geq k'+4$. See Figure \ref{fig:45sfine} for an illustration of the cases considered in this claim.

First suppose $v$ is adjacent to $u_1$. In this case, the path $v_{k'+1}u_1vv_i$ separates $G$ into two graphs $G_1$ and $G_2$ where without loss of generality $S \subseteq G_1$. (See Figure \ref{fig:45sfine}.) By Observation \ref{subcanvas}, $K[G_1]$ is an unexceptional canvas. By the minimality of $K$, it follows that $G_1$ admits an $L$-coloring $\psi$. Let $L''$ be a list assignment for $G_2$ obtained from $L$ by setting $L''(x) = \{\psi(x)\} $ for $x \in \{v_{k'+1},u_1,v,v_i\}$, and $L''(x) = L(x)$ for $x \in V(G_2) \setminus  \{v_{k'+1},u_1,v,v_i\}$. Since $\g(u_1) \geq 5$, it follows that $v_{k'+1}u_1vv_i$ is an acceptable path for $G_2$. Note that $K_2 =(G_2, L'',  v_{k'+1}u_1vv_i, A \cap V(G_2))$ is a canvas. If $K_2$ is unexceptional, it follows from the minimality of $K$ that $G_2$ admits an $L''$-colouring $\psi'$. But then $\psi \cup \psi'$ forms an $L$-colouring of $G$, a contradiction. Thus we may assume $K_2$ is an exceptional canvas.

First suppose $K_2$ is an exceptional canvas of type (i). Since $A \cap V(G_2) = \{v_{k'+2}\}$, it follows that $v_{k'+2}$ is adjacent to $v_i$. This is a contradiction, since $i \geq k'+4$ and $C$ is chordless by Lemma \ref{chordless}.

Next, suppose $K_2$ is an exceptional canvas of type (ii). Note that $\g(u_1) \geq 5$, and every vertex in a generalized wheel has girth three. Thus there exists a vertex $w \not \in \{u_1, v_{k'+1}\}$ such that $G_2$ contains a subgraph $W$ that is a generalized wheel with principal path $wvv_i$ such that the vertices on the outer cycle of $W$ are on the outer face boundary of $G_2$ and have lists of size at most three under $L''$. Moreover, there exists a vertex in $A \cap V(G_2)$ adjacent to $w$ and $v_{k'+1}$. Recall that $A \cap V(G_2) = \{v_{k'+2}\}$, and $N_G(v_{k'+2}) = \{v_{k'+1}, v_{k'+3}\}$. It follows that $w = v_{k'+3}$. But $vv_{k'+3}u_2u_1v$ is a cycle of length four and $\g(u_2) = 5$, a contradiction.

Thus we may assume that $K_2$ is an exceptional canvas of type (iii). But this too is a contradiction, since $v_{k'+1}u_1vv_i$ is a path of length three.

Thus we may assume instead that $v$ is adjacent to $u_2$. In this case, the path $v_{k'+3}u_2vv_i$ separates $G$ into two graphs $G_1$ and $G_2$ as above. Note here that $i \geq k'+5$ since $\g(u_2) \geq 5$. The argument is the same as in the previous case, except that here $V(G_2) \cap A  = \emptyset$, and so we have immediately that the canvas $(G_2, L'',  v_{k'+3}u_2vv_i, A \cap V(G_2)\})$ is unexceptional and thus that $G_2$ admits an $L''$-colouring, a contradiction.
\end{proof}

We now prove $(G', L', S, A')$ is an unexceptional canvas via the following claims.

\begin{claim} \label{3sfine}
Every vertex $v \in V(C')\setminus V(S)$ with $\g(v) = 3$ satisfies $|L'(v)| \geq 3$.
\end{claim}
\begin{proof}
Suppose not, and let $v \in V(C')\setminus V(S)$ be a counterexample. By Claim \ref{Csfine}, we have that $v \in V(\Int(C))$ and so that $|L(v)| \geq 5$. Since $|L'(v)| \leq 2$, it follows that $v$ is adjacent to at least three vertices in $V(P) \cup \{u_1, u_2\}$. It follows from Lemma \ref{P-structure-i} that $v$ is adjacent to at least one of $u_1$ and $u_2$. But this contradicts Claim \ref{notadjto2}.
\end{proof}

\begin{claim} \label{4sfine}
Every vertex $v \in V(C') \setminus V(S)$ with $\g(v) = 4$ satisfies $|L'(v)| \geq 3$.
\end{claim}
\begin{proof}
Suppose not, and let $v \in V(C')\setminus V(S)$ be a counterexample. It follows from Claim \ref{Csfine} that $v \in V(\Int(C))$ and so that $|L(v)| \geq 4$. Since $|L'(v)| \leq 2$, it follows further that $v$ is adjacent to at least two vertices in $V(P) \cup \{u_1, u_2\}$. By Lemma \ref{P-structure-i}, since $\g(v) >3$ we have that $v$ is not adjacent to two vertices in $V(P)$. Since $\g(v) = 4$, it follows that $v$ is adjacent to exactly one of $u_1$ and $u_2$ and a vertex in $V(P)$. This contradicts Claim \ref{notadjto2}.
\end{proof}

From Claims \ref{3sfine} and \ref{4sfine}, it follows that every vertex in $A'$ has girth at least five. The following lemma completes the proof that $(G', L', S, A')$ is a canvas. After this, it will remain only to show that it is unexceptional.

\begin{claim}\label{5sfine}
$A'$ is an independent set, and if $v \in A'$, then $|L'(v)| = 2$. 
\end{claim}
\begin{proof}
First, we show that every vertex in $A'$ has two colours in its list under $L'$. Suppose not; let $v \in A'$ have a list of size at most one. It follows from Claim \ref{Csfine} that $v \in A' \setminus A$, and so that $v \in V(\Int(C))$ and $|L(v)| \geq 3$. Since $|L'(v)| \leq 1$, we have that $v$ is adjacent in $G$ to at least two vertices in $V(P) \cup \{u_1, u_2\}$. Since $\g(v) \geq 5$, we have that $v$ is adjacent to at most one of $u_1$ and $u_2$. By Claim \ref{intg4s}, we have furthermore that $v$ is adjacent to at most one vertex in $V(P)$. Thus $v$ is adjacent to exactly one of $u_1$ and $u_2$, and exactly one vertex in $V(P)$. This contradicts Claim \ref{notadjto2}.

It remains to show that $A'$ is independent. Suppose not: let $v, w$ be two vertices in $A'$ with $vw \in E(G)$. Since $A$ is an independent set, it follows that at least one of $v$ and $w$ is in $A' \setminus A$. Without loss of generality, let $v \in A' \setminus A$.

First suppose $w \in A$. Note that $v$ is not adjacent to a vertex in $V(P)$, as otherwise $v$ contradicts Lemma \ref{intg4s}. Since every vertex in $A' \setminus A$ is adjacent to a vertex in $V(P) \cup \{u_1, u_2\}$, it follows that $v$ is adjacent to a vertex $u \in \{u_1, u_2\}$. If $u = u_1$, define $Q = wvu_1v_{k'+1}$. If $u = u_2$, define $Q = wvu_2v_{k'+3}$. In either case, since $\g(u_1) \geq 5$ and $\g(u_2) \geq 5$, by Lemma \ref{3-5-5-3} there exists a vertex $x \in A$ adjacent to both endpoints of $Q$. This is a contradiction, since $A$ is an independent set and $w \in A$ is an endpoint of $Q$.

Thus we may assume that both $v$ and $w$ are in the set $A' \setminus A$. Note that if each of $v$ and $w$ is adjacent to a vertex in $V(P)$, then by Lemma \ref{3-5-5-3} we have that without loss of generality $v$ is adjacent to $v_{k'+1}$ and $w$ is adjacent to $v_{k'+3}$. Since $\g(v_{k'+2}) \geq 5$, we have by Lemma \ref{sep5cycle} that $\Int(u_1v_{k'+1}v_{k'+2}v_{k'+3}u_2u_1) = \emptyset$. Thus $\{u_1, u_2\} \subseteq V(\Int(uv_{k'+1}v_{k'+2}v_{k'+3}w))$. This is a contradiction, since $\g(v_{k'+2}) \geq 5$ and so by Lemma \ref{sep5cycle} we have that $\Int(wuv_{k'+1}v_{k'+2}v_{k'+3}w) = \emptyset$. 

Thus at least one of $v$ and $w$ is adjacent to a vertex in $\{u_1, u_2\}$; moreover, since $\g(v) \geq 5$ and $\g(w) \geq 5$, exactly one of $v, w$ is adjacent to a vertex in $\{u_1, u_2\}$. Without loss of generality, we may assume $v$ is adjacent to one of $u_1$ and $u_2$, and so that there exists an index $i$ such that $w$ is adjacent to a vertex $v_i \in V(P)$. 

\vskip 4mm
\noindent
\textbf{Case 1. $i \leq k'+4$.} Note that in this case, we may assume $i = k'+4$. To see this, suppose not. First assume $v$ is adjacent to $u_2$. If $i = k'+1$, then since $\g(u_1) \geq 5$ we have by Lemma \ref{sep5cycle} that $\Int(vwv_{k'+1}u_1u_2v) = \emptyset$. Since $\Int(u_2u_1v_{k'+1}v_{k'+2}v_{k'+3}u_2) = \emptyset$ by the same lemma, it follows that $\deg(u_1) = 2$. This is a contradiction, since $|L(u_1)| \geq 3$ and $K$ is a vertex-minimum counterexample. It further follows from the fact that $\Int(u_2u_1v_{k'+1}v_{k'+2}v_{k'+3}u_2) = \emptyset$ that $i \neq k'+2$. Thus we may assume $i = k'+3$. But this is a contradiction, since $vwv_{k'+3}u_2v$ is a 4-cycle and $\g(v) \geq 5$. Symmetrical arguments show that if $v$ is adjacent to $u_1$, then $i \not \in \{k'+3, k'+2, k'+1\}$.

We may therefore assume that $i = k'+4$. If $v$ is adjacent to $u_1$, then by Lemma \ref{sep6cycle} we have that $\Int(u_1vwv_{k'+4}v_{k'+3}u_2u_1) = \emptyset$ since $\g(u_1) \geq 5$. By Lemma \ref{sep5cycle}, $\Int(u_2u_1v_{k'+1}v_{k'+2}v_{k'+3}u_2) = \emptyset$, and so $u_2$ has degree two. Since $|L(u_2)| \geq 3$ and $K$ is a vertex-minimal counterexample, this is a contradiction since every $L$-colouring of $G-u_2$ extends to an $L$-colouring of $G$.  Thus we may assume that $v$ is adjacent to $u_2$. See Figure \ref{fig:A'indep} (1) for an illustration of this case.

\begin{figure}[ht]
\tikzset{white/.style={shape=circle,draw=black,fill=white,inner sep=1pt, minimum size=7pt}}
\tikzset{starnode/.style={inner sep=1pt, minimum size=7pt, star,star points=4,star point ratio=0.5, draw, fill=white}}
\tikzset{invisible/.style={shape=circle,draw=black,fill=black,inner sep=0pt, minimum size=0.1pt}}

\begin{center}
\begin{tikzpicture}
        \node[]  at (4,8.5){\Large{\textbf{(1)}}};
        \node[white] (1) at (7,9){};
        \node[] at (7, 9.4) {$v_{k'+1}$};
        \node[starnode] (2) at (8.1,8){};
        \node[]  at (8.2, 8.4){$v_{k'+2}$};
        \node[white] (3) at (9,7){};     
        \node[]  at (9.3, 7.4){$v_{k'+3}$};
        \node[white] (4) at (10,4.5){};
        \node[]  at (10.6, 4.5){$v_{k'+4}$};
        \node[white] (6) at (6.8,5.8){}; 
        \node[]  at (6.4,5.8){$u_2$};
        \node[white] (7) at (5.7,7.25){}; 
        \node[]  at (5.3, 7.25){$u_1$};
        \node[white] (8) at (7,5.2){};  
        \node[]  at (6.6,5.1){$v$};
        \node[white] (13) at (7.25,4.5){};  
        \node[]  at (6.8,4.3){$w$};
        \node[invisible] (9) at (6,9.5){};                
        \node[invisible] (10) at (10,4){};                
        
        \draw[black] (1)--(2); 
        \draw[black] (2)--(3);
        \draw[black] (4)--(3);   
        \draw[black] (4)--(10); 
        \draw[black] (1)--(7);
        \draw[black] (1)--(9);
        \draw[black] (6)--(7);
        \draw[black] (6)--(3);  
        \draw[black] (8)--(6);   
        \draw[black] (8)--(13);
        \draw[black] (4)--(13);

\end{tikzpicture}
\vskip 6mm
\begin{tikzpicture}
        \node[]  at (4,8.5){\Large{\textbf{(2)}}};
        \node[white] (1) at (7,9){};
        \node[] at (7, 9.4) {$v_{k'+1}$};
        \node[starnode] (2) at (8.1,8){};
        \node[]  at (8.2, 8.4){$v_{k'+2}$};
        \node[white] (3) at (9,7){};     
        \node[]  at (9.3, 7.4){$v_{k'+3}$};
        \node[invisible] (11) at (9.35, 6.25){};
        \node[]  at (9.5, 6){\Large{$\ddots$}};
        \node[invisible] (12) at (9.75, 5.5){};
        \node[white] (4) at (10,4.5){};
        \node[]  at (10.5, 4.5){$v_{i-1}$};
        \node[white] (5) at (10,3.5){}; 
        \node[]  at (10.4, 3.5){$v_i$};
        \node[white] (6) at (7.25,5.65){}; 
        \node[]  at (7.25,5.25){$u_2$};
        \node[white] (7) at (5.7,7.25){}; 
        \node[]  at (5.3, 7.25){$u_1$};
        \node[white] (8) at (5,4.5){};  
        \node[]  at (4.6,4.3){$v$};
        \node[invisible] (9) at (6,9.5){};                
        \node[invisible] (10) at (10,3){};                
        \node[white] (13) at (7.25,3.5){};  
        \node[]  at (6.8,3.2){$w$};
        
        \draw[black] (1)--(2); 
        \draw[black] (2)--(3);
        \draw[black] (3)--(11);
        \draw[black] (4)--(12);   
        \draw[black] (4)--(5); 
        \draw[black] (1)--(7);
        \draw[black] (1)--(9);
        \draw[black] (6)--(7);
        \draw[black] (6)--(3);  
        \draw[black] (8)--(7);   
        \draw[black] (13)--(5);
        \draw[black] (5)--(10);
        \draw[black] (8)--(13);

\end{tikzpicture}
\begin{tikzpicture}
        \node[]  at (4,8.5){\Large{\textbf{(3)}}};
        \node[white] (1) at (7,9){};
        \node[] at (7, 9.4) {$v_{k'+1}$};
        \node[starnode] (2) at (8.1,8){};
        \node[]  at (8.2, 8.4){$v_{k'+2}$};
        \node[white] (3) at (9,7){};     
        \node[]  at (9.3, 7.4){$v_{k'+3}$};
        \node[invisible] (11) at (9.35, 6.25){};
        \node[]  at (9.5, 6){\Large{$\ddots$}};
        \node[invisible] (12) at (9.75, 5.5){};
        \node[white] (4) at (10,4.5){};
        \node[]  at (10.5, 4.5){$v_{i-1}$};
        \node[white] (5) at (10,3.5){}; 
        \node[]  at (10.4, 3.5){$v_i$};
        \node[white] (6) at (7.25,5.65){}; 
        \node[]  at (6.9,5.4){$u_2$};
        \node[white] (7) at (5.7,7.25){}; 
        \node[]  at (5.3, 7.25){$u_1$};
        \node[white] (8) at (7.5,4.5){};  
        \node[]  at (7.2,4.45){$v$};
        \node[invisible] (9) at (6,9.5){};                
        \node[invisible] (10) at (10,3){};                
        \node[white] (13) at (7.75,3.5){};  
        \node[]  at (7.3,3.45){$w$};
        
        \draw[black] (1)--(2); 
        \draw[black] (2)--(3);
        \draw[black] (3)--(11);
        \draw[black] (4)--(12);   
        \draw[black] (4)--(5); 
        \draw[black] (1)--(7);
        \draw[black] (1)--(9);
        \draw[black] (6)--(7);
        \draw[black] (6)--(3);  
        \draw[black] (8)--(6);   
        \draw[black] (13)--(5);
        \draw[black] (5)--(10);
        \draw[black] (8)--(13);

\end{tikzpicture}

\caption{Cases considered in Claim \ref{5sfine} of Lemma \ref{notctx}. Vertices in $A$ are drawn as four-pointed stars. The vertices $v$, $w$, $u_1$, and $u_2$ all have girth at least five. Recall that by Lemma \ref{sep5cycle}, since $\g(u_1)=5$ it follows that the 5-cycles in all figures have no vertices in their interior.}
    \label{fig:A'indep}
\end{center}
\end{figure}
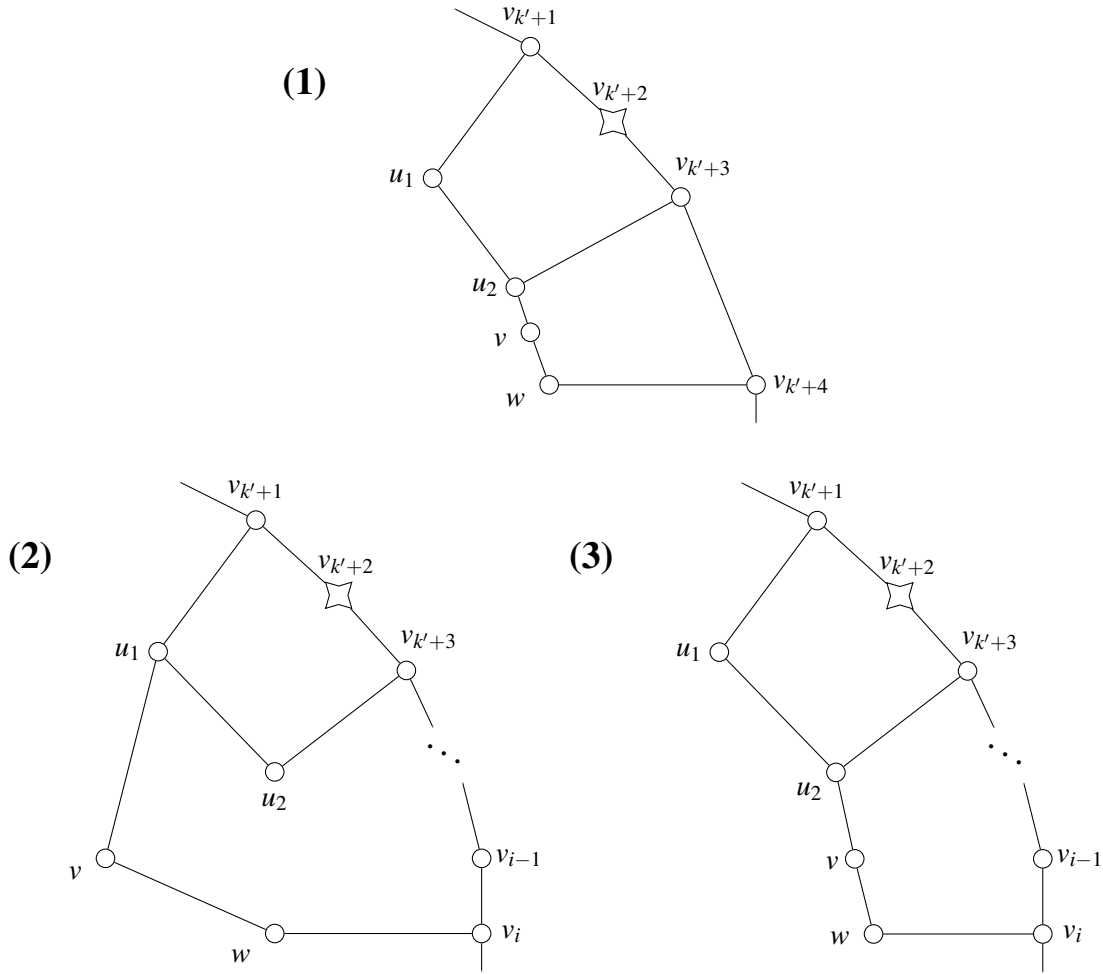

Note that by Lemma \ref{sep5cycle}, we have that $\Int(wvu_2v_{k'+3}v_{k'+4}w) = \emptyset$ since $\g(u_2) \geq 5$ and so $\deg(v_{k'+3}) = 3$ and $N_G(v_{k'+3}) = \{v_{k'+2}, u_2, v_{k'+4}\}$. Since $v_{k'+2} \in A$, we have that $\g(v_{k'+2}) \geq 5$. Furthermore, $\g(u_2) \geq 5$. It follows that $\g(v_{k'+3}) \geq 5$.

Let $\varphi$ be an $L$-colouring of $v_{k'+1}$ and $v_{k'+2}$, where $\varphi(v_{k'+1})$ is not the available colour at $v_{k'}$. Let $G'' = G - v_{k'+1} -v_{k'+2}$, and let $L''$ be a list assignment for $G''$ obtained from $L$ by setting $L''(v_{k'}) = L(v_{k'})$, and $L''(v) = L(v) \setminus \{\varphi(v_i): i \in \{k'+1, k'+2\} \textnormal{ and } v_i \in N_G(v)\}$ for all $v \in V(G'') \setminus \{v_{k'}\}$. Let $A''$ be the set of vertices in $V(G'') \setminus V(S)$ with lists of size at most two under $L''$. Note that $A'' \setminus A \subseteq (N_G(v_{k'+1}) \setminus \{v_{k'}\}) \cup \{v_{k'+3}\}$ and by definition of $P$, we have that $v_{k'+3} \in A''$. Since every vertex $v \in A'' \setminus A$ has $|L(v)| \leq 2$, it follows that $A'' \setminus (A \cup \{v_{k'+3}\})$ contains only vertices of girth at least five. Moreover, $v_{k'+3}$ also has girth at least five as mentioned above. Thus $A''$ contains only vertices of girth at least five, and so $A'' \setminus A$ is an independent set as $A'' \setminus A \subseteq N_G(v_{k'+1}) \cup \{v_{k'+3}\}$. We now argue that $A''$ is an independent set. Note that there are no edges between $A'' \setminus (A \cup \{v_{k'+3}\})$ and $A$ by Lemma \ref{intg4s}. By definition of $P$, since $i= k'+4$ we have that $v_{k'+4} \not \in A$. Moreover, since $C$ is chordless by Lemma \ref{chordless} there are no edges between $v_{k'+3}$ and vertices in $A$. Thus $A''$ is an independent set, and so  $K'' = (G'', L'', S, A'')$ is a canvas.

First suppose that $K''$ is unexceptional. Since $|V(G'')| < |V(G)|$, we have by the minimality of $K$ that $G''$ admits an $L''$-colouring $\varphi''$. But then $\varphi \cup \varphi''$ is an $L$-colouring of $G$, a contradiction. Thus we may assume $K''$ is exceptional.

Suppose that $K''$ is an exceptional canvas of type (i), and so that there exists a vertex $u \in A''$ adjacent to both $v_1$ and $v_4$. Since $K$ is unexceptional it follows that $u \in A'' \setminus A$. If $u \neq v_{k'+3}$, then $u \in N_G(v_{k'+1})$ and so $u$ contradicts Lemma \ref{intg4s}. Thus we may assume $u = v_{k'+3}$: but then $uv_4$ is a chord of $C$, contradicting Lemma \ref{chordless}.

Suppose now that $K''$ is an exceptional canvas of type (ii). Then there exists a vertex $u \in A''$ adjacent to one of $v_1$ and $v_4$ and to a vertex $y$ such that either $v_4v_3y$ or $v_1v_2y$ is the principal path of a generalized wheel $W_1$ where the vertices on the outer cycle of $W_1$ are on the outer face boundary of $G''$ and have lists of size at most three under $L''$. Note that every vertex in a generalized wheel has girth three, and for every vertex $x \in V(G'')$ with $\g_G(x) = 3$, if $|L(x)| > 3$ (and so, by Observation \ref{listsizes} if $x$ is not in $C$) then $|L(x)| \geq 5$ and so $|L''(x)| \geq 4$. It follows that every vertex in the outer cycle of $W_1$ is in the outer cycle of $G$. Since $K$ is unexceptional, we thus have that $u \in A'' \setminus A$. This is a contradiction, since there are no edges between $A'' \setminus A$ and $\{v_1, v_4\}$.

Thus we may assume that $K''$ is an exceptional canvas of type (iii) and thus that there exists a subgraph $W_2$ that is a generalized wheel with principal path $S$ such that the vertices on the outer cycle of $W_2$ are on the outer cycle of $G''$ and have lists of size at most three under $L''$. As noted above, for every vertex $x \in V(G'')$ with $\g_G(x) = 3$, if $|L(x)| > 3$ then $|L(x)| \geq 5$ and so $|L''(x)| \geq 4$. It follows that every vertex $x$ in the outer cycle of $W_2$ is in the outer cycle of $G$ and has $|L(x)| = |L''(x)|$. Since $K$ is unexceptional, this is a contradiction.

\vskip 4mm
\noindent
\textbf{Case 2.  $i \geq k'+5$.} See Figure \ref{fig:A'indep} (2), (3) for an illustration of these cases. If $v$ is adjacent to $u_1$, define $Q =  v_{k'+1}u_1vwv_i$. If $v$ is adjacent to $u_2$, define $Q =  v_{k'+3}u_2vwv_i$. In either case, the path $Q$ separates $G$ into two graphs $H_1$ and $H_2$ where without loss of generality $S \subseteq H_1$. By Observation \ref{subcanvas}, $K[H_1]$ is an unexceptional canvas. By the minimality of $K$, we therefore have that $H_1$ admits an $L$-colouring $\varphi$.

Let $L''$ be a list assignment defined by $L''(x) = L(x)$ for all $x \in V(H_2) \setminus V(Q)$, and $L''(x) = \{\varphi(x)\}$ for all $x \in V(Q)$. Note that the outer face boundary walk of $H_2$ gives a cycle formed by $Q$ and a subpath of $C$. Let $C_2$ be the outer cycle of $H_2$. By Claim \ref{P-structure-i} and the fact that $\g(w) \geq 5$, we have that $N_G(w) \cap V(P) = \{v_i\}$. Similarly, $N_G(u_1) \cap V(P) = \{v_{k'+1}\}$ and $N_G(u_2) \cap V(P) = \{v_{k'+3}\}$. Since $v$ is adjacent to one of $u_1$ and $u_2$, it follows from Claim \ref{notadjto2} that $N_G(v) \cap V(P) = \emptyset$. Finally, since $\g(x) \geq 5$ for all $x \in \{u_1, u_2, v, w\}$, it follows that $C_2$ is chordless. Since $i \geq k'+5$ and $V(P) \cap A = \{k'+2\}$, it follows from Observation \ref{listsizes} that $|L(v_{i-1})| = 3$.

Thus $H_2[V(C_2)]$ admits an $L''$-colouring $\varphi'$ obtained by extending $\varphi$ to $C_2$ by colouring the vertices of $V(C) \cap V(C_2)$ in increasing order of index. Let $H_2'$ be obtained from $H_2$ by deleting the vertices in $V(C_2) \cap V(C)$. Let $C_2'$ be the graph whose vertex- and edge-set are precisely those of the outer face boundary of $H_2'$, and let $S' = Q -C$. Let $L'''$ be the list assignment obtained from $L$ by setting $L'''(v) = L(v) \setminus \{\varphi(x): x \in N_{H_2}(v)\}$ for all $v \in V(H_2')\setminus V(S')$, and $L'''(u) = \{\varphi'(u)\}$ for all $u \in V(S')$.  Let $A_2 \subseteq V(C_2') \setminus V(S')$ be the set of vertices with lists of size at most two under $L'''$. Note that $S'$ is an acceptable path, since it has exactly three vertices. We claim $(H_2', L''', S', A_2)$ is a canvas. To see this, note that since every vertex $x \in V(C_2') \setminus V(S')$ of girth three has $|L(x)| \geq 5$ and every vertex $y \in V(C_2') \setminus V(S')$ has $|L(y)| \geq 4$, it follows from  Lemma \ref{P-structure-i} that every vertex $u \in V(C_2')\setminus V(S')$ with $\g_G(u) \in \{3,4\}$ has $|L'''(u)| \geq 3$. Thus $A_2$ contains only vertices of girth at least five in $G$. Note that $A_2 \cap A = \emptyset$, and hence $|L(u)| \geq 3$ for every $u \in A_2$. It therefore follows from Lemma \ref{P-structure-i} that every vertex in $A_2$ has a list of size exactly two under $L'''$. It remains to show that $A_2$ is an independent set. To see this, suppose not. Then there exists an edge $y_1y_2 \in E(G_2)$ with $\{y_1,y_2\} \subseteq A_2$. Note that every vertex in $A_2$ is adjacent in $G$ to a vertex in $V(P)$. Since $\g_G(y_1) \geq 5$ and $\g_G(y_2) \geq 5$, it follows from Lemma \ref{3-5-5-3} and the fact that $V(P) \cap A = \{v_{k'+2}\}$ that one of $y_1$ and $y_2$ is adjacent to $v_{k'+1}$. But since $\g_G(u_1) \geq 5$, by Lemma \ref{sep5cycle} we have that $V(\Int(u_1u_2v_{k'+3}v_{k'+2}v_{k'+1})) = \emptyset$. Thus no vertex in $A_2$ is adjacent to $v_{k'+1}$ in $G$, a contradiction.

Thus $K_2$ is a canvas. Recall that $w \in V(S')$ has girth at least five. Since $V(S') = 3$, it follows that $K_2$ is unexceptional. By the minimality of $K$, we have that $H_2'$ admits an $L'''$-colouring $\varphi''$. But then $\varphi'' \cup \varphi' \cup \varphi$ is an $L$-colouring of $G$, a contradiction. 

\end{proof}

By the previous claims, we have that $K' =(G', L', S, A')$ is a canvas. We now show it is unexceptional. 
\begin{claim}\label{unexceptional}
$K'$ is unexceptional.
\end{claim}
\begin{proof}
Suppose not. First suppose $K'$ is an exceptional canvas of type (iii).  Then $G'$ contains a generalized wheel $W$ such that the vertices on the outer cycle of $W$ are on the outer cycle of $G'$ and have lists of size at most three under $L'$. Let $v_1v_2v_3w_1 \dots w_tv_1$ be the outer cycle of $W$. Since $|L'(w_1)| = 3$, it follows that $|L(w_1)| \geq 3$. Thus $w_1 \not \in A$. Moreover, $w_1 \neq v_{k'+1}$ since $v_{k'+1} \not \in V(G')$. It follows that $w_1 \neq v_4$, and so that $w_1 \in V(\Int(C))$. Thus $|L(w_1)| \geq 5$, and so we have that $w_1$ is adjacent to at least two vertices in $V(P) \cup \{u_1, u_2\}$. Since both $u_1$ and $u_2$ have girth at least five, it follows that $w_1$ is adjacent to a vertex $v_\ell$ in $V(P)$, with $\ell \geq k'+3$. Since $w_1$ is adjacent to $v_3$, by Lemma \ref{nogenwheels} we have that $w_1$ is also adjacent to $v_{k'+1}, v_{k'+2}, \dots, v_{\ell-1}$. This contradicts the fact that $v_{k'+2} \in A$ by Lemma \ref{notnone}, and thus $\g(v_{k'+2}) \geq 5$.

Next, suppose that $K'$ is an exceptional canvas of type (i). Then there exists a vertex $u \in A'$ adjacent to $v_1$ and $v_4$. Since $K$ is unexceptional, $u \not \in A$. Thus $u \in A'\setminus A$. By Claim \ref{Csfine}, every vertex $x \in V(C) \setminus V(S)$ has $L(x) = L'(x)$, and hence $u \not \in V(C)$. It follows that $u \in V(\Int(C))$. But then $v_1uv_4$ contradicts Lemma \ref{intg4s}.

We may thus assume that $K'$ is exceptional of type (ii): and in particular, that there exists a vertex $u \in A'$ such that $u$ is adjacent to either $v_1$ or $v_4$; and such that $u$ is adjacent to a vertex $w \in V(C') \setminus V(S)$ where $v_4v_3w$  or $v_1v_2w$ is the principal path of a generalized wheel $W$ where the vertices on the outer cycle of $W$ are on the outer cycle of $G'$ and all have lists of size at most three under $L'$.  Note that since $C$ is chordless by Lemma \ref{chordless}, it follows that $w \not \in V(C)$ and so that $w \in V(\Int(C))$.

First suppose that  $u \in A$. Since $\g(w) = 3$ and $w \in V(\Int(C))$, it follows that $|L(w)| \geq 5$: thus in $G$, $w$ is adjacent to two vertices in $V(P) \cup \{u_1, u_2\}$. Note that since each of $u_1$ and $u_2$ have girth at least five, it follows that $w$ is adjacent to at most one of $u_1$ and $u_2$, and thus that $w$ is adjacent to a vertex $v_i$ in $P$. But since $u \in A$ and $w$ is adjacent to both $v_i$ and $u$, this contradicts Lemma \ref{intg4s}.

Thus we may assume that $u \in A' \setminus A$, and so that $u$ is adjacent to a vertex $x  \in V(P) \cup \{u_1, u_2\}$. If $x \in V(P)$, this contradicts Lemma \ref{intg4s} since $u$ is also adjacent to one of $v_1$ and $v_4$. Thus $u$ is adjacent to one of $u_1$ and $u_2$. 

\begin{figure}[ht]
\tikzset{white/.style={shape=circle,draw=black,fill=white,inner sep=1pt, minimum size=7pt}}
\tikzset{blacknode/.style={shape=circle,draw=black,fill=black,inner sep=1pt, minimum size=7pt}}

\tikzset{starnode/.style={inner sep=1pt, minimum size=7pt, star,star points=4,star point ratio=0.5, draw, fill=white}}
\tikzset{invisible/.style={shape=circle,draw=black,fill=black,inner sep=0pt, minimum size=0.1pt}}

\begin{center}
\hskip 6mm
\begin{tikzpicture}
        \node[invisible] (9) at (3,9.5){}; 
        \node[blacknode] (15) at (3.6,9.5){};
        \node[] at (3.6, 9.9) {$v_4$};
        \node[starnode] (14) at (5.2,9.5){};  
        \node[] at (5.2, 9.9) {$v_5$};
        \node[white] (16) at (6.4, 4.4){};     
        \node[] at (6, 4.4) {$w$};
        \node[white] (1) at (7,9){};
        \node[] at (7, 9.4) {$v_{6}$};
        \node[starnode] (2) at (8.1,8){};
        \node[]  at (8.2, 8.4){$v_{7}$};
        \node[white] (3) at (9,7){};     
        \node[]  at (9.3, 7.4){$v_{8}$};
        \node[invisible] (11) at (9.35, 6.25){};
        \node[]  at (9.5, 6){\Large{$\ddots$}};
        \node[invisible] (12) at (9.75, 5.5){};
        \node[white] (4) at (10,4.5){};
        \node[]  at (10.4, 4.5){$v_i$};
        \node[white] (6) at (7.25,5.65){}; 
        \node[]  at (7.25,5.25){$u_2$};
        \node[white] (7) at (5.7,7.25){}; 
        \node[]  at (6.1, 7.25){$u_1$};
        \node[white] (8) at (4,7.25){};  
        \node[]  at (3.7,7.2){$u$};
        \node[invisible] (10) at (10,4){};                
        
        \draw[black] (1)--(2); 
        \draw[black] (2)--(3);
        \draw[black] (3)--(11);
        \draw[black] (4)--(12);   
        \draw[black] (4)--(10); 
        \draw[black] (1)--(7);
        \draw[black] (1)--(14);
        \draw[black] (15)--(14);
        \draw[black] (15)--(9);
        \draw[black] (15)--(8);
        \draw[black] (8)--(16);
        \draw[black] (4)--(16);
        \draw[black] (6)--(7);
        \draw[black] (6)--(3);  
        \draw[black] (8)--(7);   

\end{tikzpicture}

\caption{A case considered in Claim \ref{unexceptional}. Vertices in $A$ are drawn as four-pointed stars. Both $u_1$ and $u_2$ have girth five. Recall that by Lemma \ref{sep5cycle}, since $\g(u_1)=5$ it follows that the 5-cycles in all figures have no vertices in their interior.}
    \label{fig:finalexcept}
\end{center}
\end{figure}
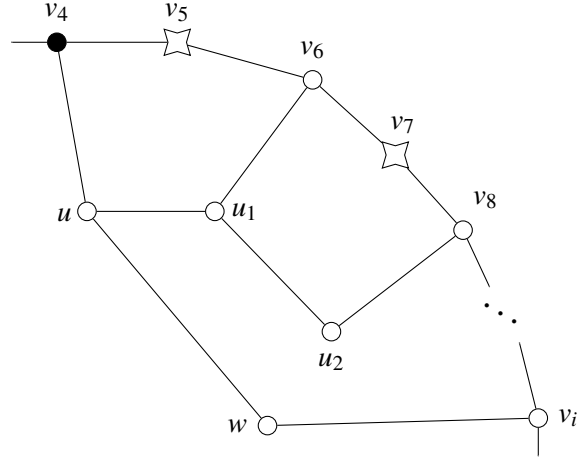

First suppose $u$ is adjacent to $u_1$. If $u$ is adjacent to $v_1$, then by Lemma \ref{3-5-5-3} applied to $u$ and $u_1$ we have that $v_1$ is adjacent to $v_{k'+2}$. Thus $q = k'+2$, a contradiction to either Lemma \ref{extraverts1} or Lemma \ref{extraverts2}. Thus we may assume that $u$ is adjacent to $v_4$, and by Lemma \ref{3-5-5-3} applied to $u$ and $u_1$, we have that $v_5 \in A$. Thus in this case $k' = k+1 = 5$. See Figure \ref{fig:finalexcept}. Note that since $|L'(w)| = 3$, we have that in $G$, $w$ is adjacent to two vertices in $V(P) \cup \{u_1, u_2\}$. Since $u_1$ has girth five and $u$ is adjacent to $u_1$, it follows that $w$ is adjacent to a vertex $v_i \in V(P)$. We claim $i \geq 8$: this follows from the fact that $u_2$ has girth five. The path $v_4uwv_i$ separates $G$ into two graphs $G_1$ and $G_2$, where without loss of generality $S \subset G_1$. By Observation \ref{subcanvas}, $K[G_1]$ is an unexceptional canvas. By the minimality of $K$, it follows that $G_1$ admits an $L$-colouring $\varphi$. Let $L''$ be the list assignment for $G_2$ obtained from $L$ by setting $L''(v) = L(v)$ for all $v \in V(G_2) \setminus \{v_4, u, w, v_i\}$ and setting $L''(v) = \{\varphi(v)\}$ for all $v \in  \{v_4, u, w, v_i\}$. Note that since $\g(u) \geq 5$, we have that $v_4uwv_i$ is an acceptable path for $G_2$. Moreover, $K_2 = (G_2, L'', v_4uwv_i, A \cap V(G_2))$ is a canvas. Since $A \cap V(G_2) = \{v_{5}, v_{7}\}$ and $i \geq 8$, it follows from the fact that $C$ is chordless (Lemma \ref{chordless}) that $K_2$ is not an exceptional canvas of type (i). Since $G$ is planar, we have that $w$ is not adjacent to $v_6$. Since $\g(u_1) \geq 5$, we have furthermore that $u$ is not adjacent to $v_6$. It follows that $K_2$ is not an exceptional canvas of type (ii). Finally, $K_2$ is trivially not an exceptional canvas of type (iii) since $v_4uwv_i$ has four vertices. Since $|V(G_2) | < |V(G)|$, it follows from the minimality of $K$ that $K_2$ admits an $L''$-colouring $\varphi''$. As $\varphi'' \cup \varphi$ is an $L$-colouring of $G$, this is a contradiction. 

We may thus assume $u$ is adjacent to $u_2$.  Recall that $w$ is adjacent to one of $v_2$ and $v_3$, and so since $G$ is planar $u$ is not adjacent to $v_1$. Thus $u$ is adjacent to $v_4$. By Lemma \ref{3-5-5-3} applied to $u_2$ and $u$, there exists a vertex in $A$ adjacent to $v_4$ and $v_{k'+3}$. Since $C$ is chordless by Lemma \ref{chordless}, this is a contradiction.
\end{proof}

Since $K$ is a minimum counterexample to Theorem \ref{345colouring} and $|V(G')| < |V(G)|$, it follows that $G'$ admits an $L'$-colouring $\phi'$. Recall that by the properties of $\phi$ described in Claim \ref{colexists}, we have that $\phi(v_{k'+1})$ is not the available colour at $v_{k'}$, and so $\phi'(v_{k'}) \neq \phi(v_{k'+1})$. But then $\phi \cup \phi'$ is an $L$-colouring of $G$, contradicting that $K$ is a counterexample to Theorem \ref{345colouring}. 
\end{proof}


\section*{Acknowledgments} 
The authors are grateful to the anonymous reviewer for helpful suggestions.

\bibliographystyle{amsplain}


\begin{aicauthors}
\begin{authorinfo}[lpos]
  Luke Postle\\
  University of Waterloo\\
  Waterloo, Ontario, Canada\\
  lpostle\imageat{}uwaterloo\imagedot{}ca \\
\end{authorinfo}
\begin{authorinfo}[esr]
  Evelyne Smith-Roberge\\
  Georgia Institute of Technology\\
  Atlanta, Georgia, USA\\
  evelyne.smithroberge\imageat{}gmail\imagedot{}com \\
\end{authorinfo}

\end{aicauthors}

\end{document}